\documentclass[final]{siamltex}

\usepackage{amsmath,amssymb,graphicx,esint}
\def\MHD{\text{\sc mhd}}
\def\A{{A}}
\def\Av{{\bf A}}
\def\B{{\bf B}}

\def\X{{\bf X}}
\def\u{{\bf u}}
\def\E{{\bf E}}
\def\x{{\bf x}}
\def\n{{\bf n}}
\def\t{{\bf t}}
\def\r{{\bf r}}

\def\r{{r}}

\def\En{{\mathcal E}}
\def\half{\frac{1}{2}}
\def\Am{{\mathcal A}}
\def\Bm{{\mathcal B}}

\newtheorem{claim}[theorem]{Claim}

\def\be{\begin{equation}}
\def\ee{\end{equation}}

\def\ba{\begin{align}}
\def\ea{\end{align}}
\def\reals{\mathbb{R}}
\def\Id{{\mathbb I}}

\title{A high-order unstaggered constrained transport method for
the 3D ideal magnetohydrodynamic equations based on the method of lines}

\author{Christiane Helzel\thanks{Department of
    Mathematics, Ruhr-University Bochum, Universit\"atsstr.\ 150,
    44780 Bochum, Germany \, ({\tt Christiane.Helzel@ruhr-uni-bochum.de})} \and 
   James A.~Rossmanith\thanks{Department of Mathematics, Iowa State University,
 396  Carver Hall, Ames, IA 50011, USA \, ({\tt rossmani@iastate.edu})} \and
  Bertram Taetz\thanks{Department of
    Mathematics, Ruhr-University Bochum, Universit\"atsstr.\ 150,
    44780 Bochum, Germany \, ({\tt Bertram.Taetz@rub.de})}}

\begin{document}

\maketitle
\begin{abstract}
Numerical methods for solving the ideal magnetohydrodynamic (MHD) 
equations in more than one space dimension must confront the challenge
of controlling errors in the discrete divergence of the magnetic field.
One approach that has been shown successful in stabilizing MHD calculations
are constrained transport (CT) schemes.  CT schemes 
can be viewed as predictor-corrector methods for updating the magnetic field, where a magnetic field value is first predicted by a method that does not exactly preserve the divergence-free condition on the magnetic field, followed by a correction step that aims to control these divergence errors.
In  Helzel et al. [{\it J. Comp. Phys.} {\bf 227}, 9527 (2011)] the authors presented an unstaggered 
constrained transport
method for the MHD equations on 
three-dimensional Cartesian grids. 
In this approach an evolution equation for the magnetic potential is
solved during each time step and a divergence-free update of the
magnetic field is computed by taking the curl of the magnetic potential. 
The evolution equation for the vector
potential is only weakly hyperbolic, which requires special
numerical treatment. A key step in this method is the use of dimensional
splitting in order to overcome these difficulties.

In this work we generalize the method  of [{\it J. Comp. Phys.} {\bf 227}, 9527 (2011)] in three
important ways: (1) we remove the need for operator splitting by switching
to an appropriate method of lines discretization and coupling this with a non-conservative
finite volume method for the magnetic vector potential equation, (2) we increase the spatial and temporal order of
accuracy of the entire method to third order, and (3) we develop the method so that it is applicable
on both Cartesian and logically rectangular mapped grids. 
The method of lines approach that is used in this work is based on a third-order accurate finite
volume discretization in space coupled to a third order strong stability preserving Runge-Kutta time stepping method.
The evolution equation for the magnetic vector potential is solved using a non-conservative finite
volume method based on the approach of Castro et al.  [{\it Math. Comput.} {\bf 79}, 1427 (2010)].
 The curl of the magnetic potential is computed via a third-order accurate discrete operator
 that is derived from appropriate application of the divergence theorem and subsequent
 numerical quadrature on element faces. Special artificial resistivity
limiters are used to control unphysical oscillations in the magnetic potential and field components
across shocks. Several test computations are shown that confirm third order accuracy for
smooth test problems and high-resolution for test problems with shock waves.

\end{abstract}



 \section{Introduction}
It is well known that numerical
methods for the multi-dimensional magnetohydrodynamic (MHD) equations must in some way respect the
divergence-free condition on the magnetic field; failure to control divergence errors have been
shown to lead to nonlinear numerical instabilities  (e.g., see
T\'oth \cite{article:To00} for a discussion).  Brackbill and Barnes \cite{article:BrBa1980} were one of
the first to document this problem; since their paper a variety of 
approaches have been proposed to control the divergence of the magnetic
field. The four main approaches that have been developed are  8-wave \cite{article:Po94,article:Po99}, 
projection \cite{article:BaKi04,article:To00}, divergence-cleaning \cite{article:DKKMSW02},
and constrained transport methods \cite{article:Ba04,article:BaSp99a,article:DaWo98,article:EvHa88,article:FeTo03,article:HeRoTa10,article:LoZa00,article:LoZa04,article:Ro04b,article:Ryu98,article:De01b,article:To00,article:To2005}.
We will not endeavor to provide a detailed description of these four numerical approaches,
since such a description already exists in T\'oth \cite{article:To00} and Helzel et al. \cite{article:HeRoTa10}.
Instead, because in this work we will develop a constrained transport method, 
we very briefly describe the CT methodology for MHD in the remainder of this section.

The basic framework of the constrained transport methodology was developed
 by Evans and Hawley \cite{article:EvHa88}. The
premise was to modify the celebrated Yee scheme \cite{article:Ye66} from vacuum electromagnetism
to the ideal MHD equations; this framework included staggered electric and magnetic fields,
just as in the Yee scheme. 
The resulting method can be viewed as predictor-corrector approach for updating the magnetic field, where a magnetic field value is first predicted by a method that does not exactly preserve the divergence-free condition on the magnetic field, followed by a correction step that aims to control these divergence errors. 
The basic steps in the Evans and Hawley 
CT approach can be outlined as follows:
\begin{description}
\item[{\bf Step 0.}] Start with MHD cell average values at  the current time: $\left(\rho^n, \,
	\rho \u^n, \, {\mathcal E}^n, \, \B^n \right)$.
\item[{\bf Step 1.}] Take a time-step using some finite volume method,
	which produces the cell average values: $\left(\rho^{n+1}, \,
	\rho \u^{n+1}, \, {\mathcal E}^{n+1}, \, \B^{\star} \right)$.
	$\B^{\star}$ is the {\it predicted} value of the magnetic field.
\item [{\bf Step 2.}] Using the ideal Ohm's law relationship, $\E = \B \times \u$,
	and  some space and time 
	interpolation scheme for $\B$ and $\u$,  reconstruct a space and time staggered electric
	field value: $\E^{n+\half}$.
\item [{\bf Step 3}.] Produce the {\it corrected} magnetic field value, $\B^{n+1}$, 
	from the Yee scheme \cite{article:Ye66} step for the magnetic field:
	\[
		\B^{n+1} = \B^n - \Delta t \, \nabla \times \E^{n+\half},
	\]
	where the curl operation is understood to be a discrete curl.
\end{description}
An alternative, but completely equivalent, interpretation of this scheme
replaces {\bf Step 3} above with the following:
\begin{description}
\item [{\bf Step 3a}.] Produce the magnetic potential value, $\Av^{n+1}$, 
	from the induction equation written in potential form:
	\[
		\Av^{n+1} = \Av^n - \Delta t \, \E^{n+\half}.
	\]
\item [{\bf Step 3b}.] Produce the {\it corrected} magnetic field value, $\B^{n+1}$, 
	by computing the discrete curl of the magnetic potential:
	\[
		\B^{n+1} = \nabla \times \Av^{n+1}.
	\]
\end{description}
By altering how {\bf Step 2} (i.e., the reconstruction of the electric field) is done in the
Evans and Hawley  framework, a variety of CT schemes were developed,
including Balsara and Spicer \cite{article:BaSp99a}, Dai and Woodward \cite{article:DaWo98},
 Ryu et al. \cite{article:Ryu98}, and Londrillo and Del Zanna \cite{article:LoZa00,article:LoZa04}. 
 In his review article, T\'oth \cite{article:To00} showed that constrained transport
methods could also be formulated without the use of staggered electric and magnetic fields. 
Several  of these unstaggered variants have been formulated, including
the schemes of Fey and Torrilhon \cite{article:FeTo03}, Helzel et al. 
\cite{article:HeRoTa10}, and Rossmanith \cite{article:Ro04b}.

The particular focus of this work is to consider modifications of the 
constrained transport approaches of Rossmanith  \cite{article:Ro04b} and Helzel et al. \cite{article:HeRoTa10}. 
Rossmanith \cite{article:Ro04b} developed an unstaggered constrained transport method for the two-dimensional MHD
equations on Cartesian grids. In this approach LeVeque's wave propagation method
\cite{article:LeVeque97,lev02} is used to update the conserved variables of the MHD
equations. 
{\bf Step 3a} of the CT framework is modified by directly solving 
a scalar transport equation for the out-of-plane magnetic potential:
\begin{equation}
	\label{eqn:A3_2D}
	A^3_{,t} + u^1 A^3_{,x} + u^2 A^3_{,y} = 0.
\end{equation}
The magnetic field components are computed via the curl operation
on the magnetic potential:
\begin{equation}
	B^1 = A^3_{,y} \quad \text{and} \quad B^2 =- A^3_{,x}.
\end{equation}
Transport equation \eqref{eqn:A3_2D} for the potential
is  solved using a modified version of LeVeque's wave propagation algorithm. 
In particular, the flux limiters of the wave propagation
method are modified in the numerical solution of \eqref{eqn:A3_2D} in
order to avoid not only unphysical oscillations in  the magnetic
potential but also in derivatives of the potential (i.e., the magnetic field). 

Helzel et al. \cite{article:HeRoTa10} extended this unstaggered
constrained transport method to the three-dimensional MHD equations. In this case one
has to consider the following transport equation for the vector potential:
\begin{equation}
\begin{split}
\begin{bmatrix}
  \A^1 \\ \A^2 \\ \A^3
\end{bmatrix}_{,t} 
& +
\begin{bmatrix}
  0 & -u^2 & -u^3 \\
  0 & u^1 & 0 \\
  0 & 0 & u^1
\end{bmatrix}
\begin{bmatrix}
  \A^1 \\ \A^2 \\ \A^3
\end{bmatrix}_{,x}
 +
\begin{bmatrix}
   u^2 & 0 & 0 \\
  -u^1 & 0 & -u^3 \\
  0 & 0 & u^2
\end{bmatrix}
\begin{bmatrix}
  \A^1 \\ \A^2 \\ \A^3
\end{bmatrix}_{,y}
 \\ 
 & + \begin{bmatrix}
  u^3 &  0 &  0 \\
  0 & u^3 & 0 \\
  -u^1 &  -u^2 &  0
\end{bmatrix}
\begin{bmatrix}
  \A^1 \\ \A^2 \\ \A^3
\end{bmatrix}_{,z} = -
\begin{bmatrix}
	\psi_{,x} \\  \psi_{,y}  \\  \psi_{,z}
\end{bmatrix}.
\end{split}
\end{equation}
The magnetic field components are computed via the curl operation
on the magnetic potential:
\begin{equation}
	B^1 = A^3_{,y} - A^2_{,z}, \quad B^2 = A^1_{,z} - A^3_{,x}, 
		\quad \text{and} \quad B^3 = A^2_{,x} - A^1_{,y}.
\end{equation}
 In contrast to the 2D case, the 3D transport equation for the magnetic potential needs an additional
 constraint (i.e., the gauge condition) in order to get a closed set of evolution equations for the
 vector potential ${\bf A}$ and the scalar potential $\psi$.
Helzel et al. \cite{article:HeRoTa10} chose the so-called Weyl gauge ($\psi \equiv 0$), 
which results in a transport equation for the vector potential that is only weakly
hyperbolic. This means that standard methods for hyperbolic problems 
that are based on an eigenvector decomposition of the jump of the advected quantities    
at grid cell interfaces cannot be used. In order to overcome this difficulty,
Helzel et al. \cite{article:HeRoTa10} developed a second order accurate dimensional split method for the
magnetic vector potential equation that
is well defined also in the weakly hyperbolic case. 
Many test computations confirmed that this approach leads to a
robust and accurate approximation of the MHD equations on Cartesian
grids.  For smooth solutions of the MHD equations, the method from
\cite{article:HeRoTa10} was shown to be second order accurate.
 
The focus of this work is to generalize the approach \cite{article:HeRoTa10} in three
important ways: 
\begin{enumerate}
\item We remove the need for operator splitting by switching
to an appropriate method of lines discretization and coupling this with a non-conservative
finite volume method for the magnetic vector potential equation.
\item We increase the spatial and temporal order of
accuracy of the entire method to third order using a third-order SSP-RK time-stepping
method and a third-order finite volume spatial discretization.
\item We develop the method so that it is applicable
on both Cartesian and logically rectangular mapped grids. 
\end{enumerate}
The nonconservative method that we use on the magnetic vector potential
equation is inspired by previous work on
the approximation of non-conservative hyperbolic systems by
Canestrelli et al.\ \cite{dumb09_1D}, Dumbser et al.\
\cite{dumb10_2D}, and Ketcheson et al. \cite{sharpclaw}. 
By using a method of lines approach with a stong stability preserving
Runge-Kutta (SSP-RK) method in time and third order accurate
reconstruction in space, we are able to construct a third order
accurate unstaggered constrained transport method. In this approach,
the correction of the magnetic field is performed during each stage of
the Runge-Kutta method. For the spatial discretization we use a multidimensional 
(limited) polynomial reconstruction of the cell average quantities.
%

After a brief review of the ideal MHD equations in Section \ref{sec:mhd},
we introduce the proposed constrained transport scheme in a series of four sections.
In Section \ref{section:temp} we outline the time discretization of the proposed method
and show how to embed the constrained transport steps into the various
stages of the third-order SSP Runge-Kutta scheme.
In Section \ref{section:mhd_space} we describe the third order
accurate spatial discretization for the MHD system. 
In Section \ref{section:potential_space} we show how to modify this
approach to obtain a high-order non-conservative spatial discretization
of the magnetic vector potential equation. In this section we also
develop special artificial resistivity limiters that are used to control
unphysical oscillations in both the magnetic potential and magnetic field components.
In Section \ref{sec:curlA}
we show how to compute the curl to high-order of the
discrete magnetic vector potential.
Several numerical examples on Cartesian and mapped grids are
shown in Section \ref{section:test-computations}; these examples are used to both
verify the third-order accuracy in space and time, as well as the
shock-capturing ability of the proposed scheme. Conclusions can be
found in Section \ref{sec:conclusions}.

\section{The ideal MHD equations} 
\label{sec:mhd}
The ideal MHD equations are a first order hyperbolic system of
conservation laws that can be written in the form
\begin{gather}
\label{MHD_eqn}
 \frac{\partial}{\partial t}
  \begin{bmatrix}
    \rho \\ \rho \u \\ \En \\ \B
  \end{bmatrix} +  \nabla \cdot
  \begin{bmatrix} \rho \u \\ \rho \u  \u + \left( {p} + \frac{1}{2}
  \| \B \|^2  \right) {\mathbb I}
    - \B  \B \\
    \u \left(\En + {p} + \frac{1}{2} \| \B \|^2 \right) - 
    	\B \left(\u \cdot \B \right)
    \\ \u  \B - \B  \u
  \end{bmatrix} = 0, \\
  \label{eqn:divfree}
  \nabla \cdot \B = 0,
\end{gather}
where $\rho$, $\rho \u$ and $\En$ are the total mass, momentum and
energy densities, and $\B$ is the magnetic field. The termal pressure,
$p$, is related to the conserved quantities through the ideal gas law
\be
\label{eqn:eos}
p = (\gamma -1) \left( \En - \frac{1}{2} \| \B \|^2 - \frac{1}{2} \rho
  \| \u \|^2 \right),
\ee 
where $\gamma = 5/3$ is the ideal gas constant. Here $\| \cdot \|$
denotes the Euclidean vector norm.
A complete derivation and discussion
of MHD system \eqref{MHD_eqn}--\eqref{eqn:divfree} can
be found in several standard plasma physics textbooks
(e.g., \cite{book:Chen84,book:Go98,book:Pa91}).

\subsection{Hyperbolicity of the MHD system}
We first note that system \eqref{MHD_eqn}, along with the
equation of state \eqref{eqn:eos}, provides
a full set of equations for the time evolution 
of all eight state variables: $\left( \rho, \, \rho \u, \, \En, \, \B \right)$.
These evolution equations form a hyperbolic system. In particular, the
eigenvalues of the flux Jacobian in some arbitrary
direction ${\bf n}$ ($\| {\bf n} \| = 1$) can be written as follows:
  \begin{alignat}{2}
  \lambda^{1,8} &= {\bf u} \cdot {\bf n} \mp c_f & &\text{ :  fast magnetosonic waves,} \\
  \lambda^{2,7} &= {\bf u} \cdot {\bf n}  \mp c_a & &\text{ :  Alfv$\acute{\text{e}}$n waves,} \\
  \lambda^{3,6} &= {\bf u} \cdot {\bf n}  \mp c_s & &\text{ :  slow magnetosonic waves,} \\
  \lambda^{4} &= {\bf u} \cdot {\bf n} & &\text{ :  entropy wave,} \\
  \lambda^{5} &= {\bf u} \cdot {\bf n}  & &\text{ :  divergence wave,}
\end{alignat}
where 
\begin{align}
  a &\equiv \sqrt{\frac{\gamma p}{\rho}}, \\
  c_a &\equiv \sqrt{\frac{\left(\B \cdot \n \right)^2}{\rho}}, \\
  c_f &\equiv \left\{ \frac{1}{2} \left[ a^2 +
      \frac{\|\B\|^{2}}{\rho} + \sqrt{\left(a^2 +
      \frac{\|\B\|^{2}}{\rho} \right)^{2} - 4 a^2
      \frac{\left(\B \cdot \n \right)^2}{\rho}} \right] \right\}^{1/2}, \\
  c_s &\equiv \left\{ \frac{1}{2} \left[ a^2 +
      \frac{\|\B\|^{2}}{\rho} - \sqrt{\left(a^2 +
      \frac{\|\B\|^{2}}{\rho} \right)^{2} - 4 a^2
      \frac{\left(\B \cdot \n \right)^2}{\rho}} \right] \right\}^{1/2} \, .
\end{align}
The eigenvalues are well-ordered in the sense that
\begin{equation}
  \lambda^{1} \le \lambda^{2} \le \lambda^{3} \le \lambda^{4} \le \lambda^{5} \le \lambda^{6}
  \le \lambda^{7} \le \lambda^{8} \, .
\end{equation}
The fast and slow magnetosonic waves are genuinely nonlinear, while the remaining waves are linearly degenerate.
Note that the so-called {\it divergence-wave} has been
made to travel at the speed $\u \cdot {\bf n}$ 
via the Godunov-Powell formulation  \cite{article:Go72,article:Po94,article:Po99},
thus restoring Galilean invariance.

\subsection{The role of $\nabla \cdot \B=0$ in numerical discretizations}
If the initial magnetic field is divergence-free, then
under the evolution described by equation \eqref{MHD_eqn}, the
divergence-free condition is satisfied for all time.  This makes
\eqref{eqn:divfree} an {\it involution} \cite{book:Dafermos10}. 
Unfortunately, although \eqref{eqn:divfree}
is an involution under the evolution of the exact MHD system, it is generally
not an involution of the discretized MHD equations.
Furthermore, it has been well-documented in the literature that the
failure of numerical methods to control divergence errors can lead to
nonlinear numerical instabilities \cite{article:To00}.

Explanations for why the divergence errors lead to 
numerical instabilities have been considered from a few
different points-of-view. We briefly review the two main
explanations below.
\begin{enumerate}
\item Brackbill and Barnes \cite{article:BrBa80} showed
that if $\nabla \cdot \B \ne 0$, then the magnetic force,
\be
	{\bf F} = \nabla \cdot \left\{ \B \B - \frac{1}{2} \| \B \|^2 \, {\mathbb I}
		\right\},
\ee
 in the direction of the
   magnetic field, will not in general vanish:
   \be  {\bf F} \cdot \B 
   	= \| \B \|^2 \, \left( \nabla \cdot \B\right) \ne 0. 
\ee
  If this spurious forcing becomes too large, it can lead to numerical instabilities \cite{article:BrBa80,article:Ro04b,article:To00}.
  
\item Barth \cite{article:Barth05} gave an explanation based on the well-known result of Godunov \cite{article:Go72}
that the MHD entropy density,
\be
U(q) = -\rho \log \left( p \rho^{-\gamma} \right),
\ee
produces a set of {\it entropy variables}, $v=U_{,q}$, that do not immediately
symmetrize the ideal MHD equations. Instead, a symmetric hyperbolic
form of ideal MHD can only be obtained if an additional term
that is proportional to the divergence of the magnetic field
is included in the MHD equations:
\be
\underset{\text{ideal MHD}}{\underbrace{q_{,t} + {\nabla} \cdot {\bf F}(q)}} +  
		\underset{\text{additional term}}{\underbrace{\chi_{,v} \,
			{\nabla} \cdot \B}} = 0, \qquad \text{where} \quad
	\chi(q) = \left( \gamma -1 \right) \frac{\rho \u \cdot \B}{p}.
\ee
By looking at how the entropy behaves on the discrete
level, Barth \cite{article:Barth05} was able to prove that certain
discontinuous Galerkin discretizations of the ideal MHD
equtions could be made to be {\it entropy stable} (see
Tadmor \cite{article:Tad03})
 if the discrete magnetic field was made globally divergence-free.
The implication of this result is that schemes that do not
control errors in the divergence of the magnetic field
run the risk of becoming entropy unstable.
\end{enumerate}

\subsection{The evolution of the magnetic potential}
Since the magnetic field is divergence-free, it can be written as the
curl of a magnetic vector potential:
\be
	\B = \nabla \times \Av.
\ee
The essential feature of constrained transport methods is to
update the magnetic vector potential in some manner, and then
by computing from it some discrete form of the curl,  to obtain a
discrete divergence-free magnetic field.
Using the relation
$$\nabla \cdot \left( \u \B - \B \u \right) = \nabla \times \left( \B
  \times \u \right),$$
the last row of (\ref{MHD_eqn}) can be written in the form
\be
\label{eqn:induction}
\B_{,t} + \nabla \times \left( \B \times \u \right) = 0.
\ee
Since $\B$ is divergence free, we set $\B = \nabla \times \Av$ and rewrite \eqref{eqn:induction} as
\begin{gather}
\nabla \times \left\{  \Av_{,t}  + \left( \nabla \times \Av \right) \times \u \right\} = 0, \\
\label{eqn:vecpot_general}
\Longrightarrow \quad \Av_{,t} +  \left( \nabla \times \Av \right) \times \u = -\nabla \psi,
\end{gather}
where $\psi$ is an arbitrary scalar function. Different choices of $\psi$
represent different {\it gauge condition} choices, see
\cite{article:HeRoTa10}.

Our approach is based on the Weyl gauge, which means we take $\psi \equiv
0$. This results in an evolution equation for the vector potential is given by
\begin{equation}\label{eqn:potential}
\Av_{,t} +  \left( \nabla \times \Av \right) \times \u = 0,
\end{equation}
which can also be rewritten in quasilinear form as
\begin{equation}
\label{eqn:system_compact}
\Av_{,t} + N_1(\u) \, \Av_{,x} + N_2(\u) \, \Av_{,y} + N_3(\u) \, \Av_{,z} = 0,
\end{equation}
with 
\begin{equation}\label{eqn:N1N2N3}
N_1 = \begin{bmatrix}
0 & -u^2 & -u^3\\
0 & u^1 & 0\\
0 & 0 & u^1\end{bmatrix},  \quad
N_2 = \begin{bmatrix}
u^2 & 0 & 0\\
-u^1 & 0  & -u^3\\
0 & 0 & u^2\end{bmatrix}, \quad
N_3 = \begin{bmatrix}
u^3 & 0  & 0\\
0 & u^3 & 0\\
-u^1 & -u^2 & 0 \end{bmatrix}.
\end{equation}
The the flux Jacobian in some
direction $\n = (n^1,n^2,n^3)^T$,
\begin{equation}
\label{eqn:fluxJ}
M(\n,\u) := n^1 N_1 + n^2 N_2 + n^3 N_3=
\begin{bmatrix}
n^2 u^2 + n^3 u^3 & - n^1 u^2 & - n^1 u^3 \\
-n^2 u^1 & n^1 u^1 + n^3 u^3 & - n^2 u^3 \\
-n^3 u^1 & -n^3 u^2 & n^1 u^1 + n^2 u^2
\end{bmatrix},
\end{equation}
has real eigenvalues for all
$\|\n\|=1$, but for any nonzero velocity vector, $\u$, there exist directions $\n$, for which $M(\n,\u)$
does not have a complete set of right eigenvectors. Therefore, system \eqref{eqn:system_compact}--\eqref{eqn:N1N2N3}
 is only {\em weakly hyperbolic} \cite{article:HeRoTa10}.
In order to see this weak hyperbolicity, we write the eigenvalues of matrix \eqref{eqn:fluxJ}: 
\be
\lambda = \bigl\{0, \n \cdot \u, \n \cdot \u\bigr\},
\ee
and the matrix of right eigenvectors:
\begin{equation}
R = \Biggl[ \r^{(1)} \, \Biggl| \, \r^{(2)} \, \Biggl| \, \r^{(3)} \Biggr] = 
\begin{bmatrix}
n^1 & \quad  n^2 u^3 - n^3 u^2 & \quad u^1 \left( \u \cdot \n \right) - n^1 \| \u \|^2 \\
n^2 & \quad n^3 u^1 - n^1 u^3 & \quad u^2 \left( \u \cdot \n \right) - n^2 \| \u \|^2 \\
n^3 & \quad n^1 u^2 - n^2 u^1 & \quad u^3 \left( \u \cdot \n \right) - n^3 \| \u \|^2
\end{bmatrix}.
\end{equation}
Assuming that $\| \u \| \ne 0$ and $\| \n \| = 1$, the determinant of
matrix $R$ can be written as
\begin{equation}
\text{det}(R) = -  \| \u \|^3 \, \cos(\alpha) \, \sin^2(\alpha),
\end{equation}
where $\alpha$ is the angle between the vectors $\n$ and $\u$.
The difficulty is that for any non-zero velocity vector, ${\bf u}$, 
one can always find four directions, $\alpha=0$, $\pi/2$, $\pi$, and $3\pi/2$,
such that $\text{det}(R) = 0$.
In other words, for every $\| {\bf u} \| \ne 0$ there exists four degenerate directions in which the eigenvectors are incomplete.

 \section{Outline of the constrained transport
   algorithm and temporal discretization}
\label{section:temp}
In our constrained transport method for the MHD equations, we separate
the discretization in space and time. For the temporal discretization
we use an SSP-RK method of order up to three, see e.g.\ \cite{shu_01,book:ssp2011}.

Consider a system of ordinary differential equations of the form 
\begin{equation}\label{ode1}
Q'(t) = {\cal L}(Q(t)).
\end{equation}
One time step of the third order accurate SSP-RK method can be written
in the form
\begin{equation}
\begin{split}
Q^{(1)} & = Q^n + \Delta t {\cal L}(Q^n), \\
Q^{(2)}  & = \frac{3}{4} Q^n + \frac{1}{4} Q^{(1)} + \frac{1}{4}
\Delta t {\cal L}(Q^{(1)}), \\
Q^{n+1} & = \frac{1}{3} Q^n + \frac{2}{3} Q^{(2)} + \frac{2}{3} \Delta
t {\cal L}(Q^{(2)}). 
\end{split}
\end{equation}

In a method of lines approach we obtain systems of ordinary
differential equations of the form (\ref{ode1}) after discretizing the
spatial derivatives. Here the components of the vector $Q(t)$
represent cell average values of the physical quantities in the
different grid cells.

The semi-discrete form of  (\ref{MHD_eqn}) now has the form
\begin{equation}\label{MHD_semi}
Q'_{\MHD}(t) = {\cal L}_1(Q_{\MHD}(t)),
\end{equation}
where $Q_{\MHD}(t)$ represents the grid function at time $t$ consisting
of all cell averaged values of the conserved quantities form the
MHD equations: $(\rho,\rho \u, \En, \B)$. The precise form of the spatial discretization
represented by ${\cal L}_1(Q_{\MHD}(t)$ will be discussed in Section
\ref{section:mhd_space}. The semi-discrete form of the evolution equation for
the potential (\ref{eqn:potential}) has the general form
\begin{equation}\label{potential_semi}
Q'_{\Av} (t) = {\cal L}_2 ( Q_{\Av}(t),
Q_{\MHD}(t) ),
\end{equation}
where $Q_{\Av}(t)$ is the grid function at time $t$ consisting of the cell
averaged values of the magnetic vector potential, $\Av$.
Note that the evolution of the potential depends on the velocity field, which we take
to be as given from the solution step of the MHD equation. This is reflected in the notation
used in (\ref{potential_semi}). The spatial discretization of the
potential equation will be discussed in Section \ref{section:potential_space}. 

A single stage the constrained transport algorithm has the form:
\begin{enumerate}
\item[0.] Start with $Q_{\MHD}^n$ and $Q_{\Av}^n$ (the solution from the previous time step) and
	$Q_{\MHD}^{(k-1)}$ and $Q_{\Av}^{(k-1)}$ (the solution from the previous time stage).

\item[1.] Update without regard to the divergence-free constraint on the magnetic field:
\begin{equation*}
\begin{split}
Q_{\MHD}^{(k\star)} & = \alpha^{(k)} Q_{\MHD}^{n} + \left(1-\alpha^{(k)} \right)  Q_{\MHD}^{(k-1)} + \beta^{(k)}  \Delta t \, {\cal L}_1 \left(Q_{\MHD}^{(k-1)} \right),\\
Q_{\Av}^{(k)} & =  \alpha^{(k)} Q_{\Av}^{n} + \left(1-\alpha^{(k)} \right)  Q_{\Av}^{(k-1)} +  \beta^{(k)}  \Delta t \, {\cal L}_2 \left(Q_{\Av}^{(k-1)},
Q_{\MHD}^{(k-1)} \right),
\end{split}
\end{equation*}
where $Q_{\MHD}^{(k\star)} = \left( \rho^{(k)}, \rho \u^{(k)}, \, \En^{(k)}, \,
\B^{(k\star)} \right)$ and $\B^{(k\star)}$ denotes the {\it predicted} value of the
magnetic field in the first Runge-Kutta stage.

\item [2.] The magnetic field components of $Q_{\MHD}^{(k\star)}$  are then {\it corrected} by 
$\nabla \times Q_{\Av}^{(k)}$:
\begin{equation*}
	\B^{(k)} = \nabla \times Q_{\Av}^{(k)} \quad \Longrightarrow \quad
	Q_{\MHD}^{(k)}
=  \left( \rho^{(k)}, \rho \u^{(k)}, \, \En^{(k)}, \,
\B^{(k)} \right).
\end{equation*}
\end{enumerate}
%
For smooth solutions, this procedure with third-order accurate SSP-RK gives an update of the grid function $Q_{\MHD}^{n+1}$ that is
$3^{\text{rd}}$ order accurate in time.
This fact is confirmed numerically via  test computations done in Section \ref{section:test-computations}.

Note that in each stage we take one Euler step on the magnetic vector potential
equation \eqref{eqn:potential}. In each of these steps we evaluate the
spatial discretization operator, ${\cal L}_2$, at the current values of 
the potential, $Q_{\Av}$, and the current velocity values that can be obtained
from $Q_{\MHD}$. Just as in the constrained transport approaches of  \cite{article:HeRoTa10,article:Ro04b},
during the Euler step on \eqref{eqn:potential} we view the velocity $\u$ as a given function,
and view \eqref{eqn:potential} as a closed equation for the magnetic potential $\Av$.

\section{Spatial discretization of the MHD equations}
\label{section:mhd_space}
Our spatial discretization of the MHD equations is similar to the one
used in {\sc sharpclaw} by Ketcheson et al.\ \cite{sharpclaw}. However,
in our method we use a multidimensional polynomial
reconstruction. 
This gives the full order of convergence (here up to third order) 
for smooth nonlinear problems on Cartesian grids and reduces grid
effects for non-smoothly varying mapped grids. 
We describe the one-dimensional, multidimensional Cartesian, and 
multidimensional mapped grid spatial discretizations in
the subsequent three subsections.

\subsection{One-dimensional spatial discretization}
We first briefly discuss the one-dimen\-sio\-nal case, where we consider a
hyperbolic conservation law of the form 
\begin{equation}\label{hyp_con_law}
q_{,t}  + f(q)_{,x} = 0,
\end{equation}
where $q:\mathbb{R}^+ \times \mathbb{R} \rightarrow \mathbb{R}^m$ is a
vector of conserved quantities and $f: \mathbb{R}^m \rightarrow
\mathbb{R}^m$ is the flux function. For smooth solutions
we can transform (\ref{hyp_con_law}) into the equivalent quasilinear form
\begin{equation}
q_{,t} + A(q) \, q_{,x} = 0,
\end{equation}
where $A(q) := f_{,q}(q)$ is the flux Jacobian matrix ($A:\reals^m \rightarrow \reals^{m \times m}$). 
We assume that some initial data
$q(0,x) = q_0(x)$ with $q_0 : \mathbb{R} \rightarrow \mathbb{R}^m$ 
and some appropriate boundary conditions are given.

The cell average of the conserved quantity in grid cell $i$ at time
$t$ is denoted by $Q_i(t)$, i.e.\
\be
Q_i(t) \approx \frac{1}{\Delta x}
\int_{x_{i-\frac{1}{2}}}^{x_{i+\frac{1}{2}}} q(t,x) \, dx.
\ee
From the given  cell average values we compute piecewise polynomial
approximations (with WENO limiting)  
of the conserved quantities, using polynomials of degree at most $p$.
For smooth solutions these polynomial
agree up to order $p$ with the exact solution:
\begin{equation}\label{eqn:tildeq}
\begin{split}
\tilde{q}(t^n,x) & := \tilde{q}_i (t^n,x) \quad \text{for} \quad x \in
\left( x_{i-\frac{1}{2}},x_{i+\frac{1}{2}} \right),\\
\tilde{q}_i(t^n,x) & = q(t^n,x) + {\cal O}\left(\Delta x^{p+1} \right).
\end{split}
\end{equation} 
The reconstructed values from grid cell $i$ at the left and right grid
cell interface are denoted by
\begin{equation}\label{qpm}
q_{i-\frac{1}{2}}^+ := \lim_{\varepsilon\rightarrow 0} \, \tilde{q}_i \left(x_{i-\frac{1}{2}} + \varepsilon \right) \quad
\text{and} \quad
q_{i+\frac{1}{2}}^- := \lim_{\varepsilon\rightarrow 0} \, \tilde{q}_i \left(x_{i+\frac{1}{2}} - \varepsilon \right).
\end{equation}

In semi-discrete form, the method that is used to evolve the cell
averages of the conserved quantities is given by (see \cite{sharpclaw})
\begin{equation}\label{method1}
Q'_i(t) = -\frac{1}{\Delta x} \left( {\cal
    A}^- \Delta q_{i+\frac{1}{2}} + {\cal A}^+ \Delta
  q_{i-\frac{1}{2}} + {\cal A}\Delta q_i \right).
\end{equation}
The fluctuations ${\cal A}^\pm \Delta q_{i+\frac{1}{2}}$ can be
computed as in the standard wave propagation algorithm of LeVeque
\cite{article:LeVeque97}, or by using the $f$-wave approach of Bale et
al.\ \cite{article:BLMR2002},
with the only difference being that the left and
right values of the conserved quantities used in the wave 
decomposition or $f$-wave decomposition are obtained from the 
reconstructed interface values instead of the cell average values. 
The additional term 
$$
{\cal A} \Delta q_i =
\int_{x_{i-\frac{1}{2}}}^{x_{i+\frac{1}{2}}} A(\tilde q_i) \,
\tilde{q}_{i,x} \, dx \approx \int_{x_{i-\frac{1}{2}}}^{x_{i+\frac{1}{2}}} A(q)
\, {q}_{,x} \, dx,
$$
can be computed via Gaussian quadrature of the appropriate order. For the
approximation of a hyperbolic problem in the flux difference 
form (\ref{hyp_con_law}), the integral is equal to the flux
difference:
$$
\int_{x_{i-\frac{1}{2}}}^{x_{i+\frac{1}{2}}} A(\tilde{q}_i)
\, \tilde{q}_{i,x} \, dx = f\left(q_{i+\frac{1}{2}}^- \right) - f\left(q_{i-\frac{1}{2}}^+ \right).
$$
Furthermore, the fluctuations satisfy
\begin{equation*}
\begin{split}
{\cal A}^+\Delta q_{i-\frac{1}{2}} & = f\left(q_{i-\frac{1}{2}}^+\right) -
f_{i-\frac{1}{2}}^*\left(q_{i-\frac{1}{2}}^-,q_{i-\frac{1}{2}}^+\right), \\
{\cal A}^- \Delta q_{i+\frac{1}{2}} & = f_{i+\frac{1}{2}}^*\left(q_{i+\frac{1}{2}}^-,q_{i+\frac{1}{2}}^+\right) -
f\left(q_{i+\frac{1}{2}}^-\right).
\end{split}
\end{equation*}
Here $f^*$ denotes the interface flux, which is computed by solving a
Riemann problem for the reconstructed left and right interface values.
Thus for hyperbolic problems in divergence form,
 the method can be formulated as a  finite volume method of the form
\begin{equation}\label{method-fv}
 Q'_i(t)  = \frac{1}{\Delta x} \left(
  f_{i+\frac{1}{2}}^*\left(q_{i+\frac{1}{2}}^-,q_{i+\frac{1}{2}}^+\right) - f_{i-\frac{1}{2}}^*\left(q_{i-\frac{1}{2}}^-,q_{i-\frac{1}{2}}^+\right) \right).
\end{equation}
The Riemann problem can be solved using a variety of
solvers; as a matter of practice, in this work we use the Roe-type
Riemann solvers of the form described in Bale et al. \cite{article:BLMR2002}.

\subsection{Multidimensional Cartesian grids}
\label{multiD_hyp}
For simplicity, we consider a two-dimensional hyperbolic equation of the
form
\begin{equation}
q_{,t} + f(q)_{,x} + g(q)_{,y} = 0,
\end{equation}
or of the quasilinear form
\begin{equation}
q_{,t} + A(q) \, q_{,x} + B(q) \, q_{,y} = 0.
\end{equation}
Here  $q: \mathbb{R}^+ \times \mathbb{R}^2  \rightarrow \mathbb{R}^m$
is a vector of conserved quantities,
$f, g: \mathbb{R}^m \rightarrow \mathbb{R}^m$ are numerical flux
functions, and $A(q)=f'(q)$, $B(q)=g'(q)$ are the flux Jacobian
matrices. 

In order to get a high-order accurate approximation in space with the
method of lines approach, we need to reconstruct values on
cell interfaces at nodes of a quadrature formula that is used 
to approximate interface fluxes (or fluctuations).
In order to do this, we compute a piecewise polynomial reconstruction
$\tilde{q}$ of the conserved quantities in each grid cell.
The grid cell average of this reconstructed polynomial agrees with the
original cell average (this is needed to guarantee numerical
conservation). Furthermore, the reconstruction  is
based  on a (limited) least
squares approach that includes all neighbors that share an edge or a
corner with the considered grid cell, see \cite{article:BF1990,ti_10,ts_11}.   

The numerical method can be written as a finite volume scheme in the 
semi-discrete form 
\begin{equation}
Q'_{ij}(t) = - \frac{1}{\Delta x} \left(
  f_{i+\frac{1}{2} \, j}^* - f_{i-\frac{1}{2} \, j}^*\right) -
\frac{1}{\Delta y} \left( g_{i \, j+\frac{1}{2}}^* -
  g_{i \, j-\frac{1}{2}}^* \right),
\end{equation}
or in the form used by the wave propagation algorithm (compare with \cite{sharpclaw}) 
\begin{equation}
\begin{split}
Q'_{ij}(t) =  &\frac{1}{\Delta x} \left(
  {\cal A}^+ \Delta q_{i-\frac{1}{2} \, j} + {\cal A}^- \Delta
  q_{i+\frac{1}{2} \, j} + {\cal A} \Delta q_{ij} \right)\\
+ &\frac{1}{\Delta y} \left( {\cal B}^+ \Delta q_{i \, j-\frac{1}{2}} +
  {\cal B}^- \Delta q_{i \, j+\frac{1}{2}} + {\cal B} \Delta q_{ij}
\right).
\end{split}
\end{equation}
Now $Q_{ij}(t)$ is an approximation of the cell average of the
conserved quantity at time $t$  in grid cell $(i,j)$. 
Using the piecewise polynomial reconstruction of the conserved
quantity $q$ in each grid cell, we can evaluate left and
right values of the conserved quantity at Gaussian quadrature points 
along  grid cell interfaces.
These are used to obtain a high order accurate representation of the
flux:
\begin{align}
\frac{1}{\Delta y} \int_{y_{j-\frac{1}{2}}}^{y_{j+\frac{1}{2}}}
f^*\left(q\left(x_{i\pm\frac{1}{2}},y \right) \right) \, dy  \approx
\frac{1}{\Delta y} \sum_{k=1}^q c_k \,
f^*\left(q_{i\pm\frac{1}{2} \, j_k}^-,q_{i\pm\frac{1}{2} \, j_k}^+\right) 
 &=: f_{i\pm \frac{1}{2} \, j}^*, \\
\frac{1}{\Delta x} \int_{x_{i-\frac{1}{2}}}^{x_{i+\frac{1}{2}}}
g^*\left(q\left(x,y_{j\pm\frac{1}{2}}\right)\right) \, dx  \approx
\frac{1}{\Delta x} \sum_{k=1}^q c_k \,
g^*\left(q_{i_k \, j\pm \frac{1}{2}}^-,q_{i_k \, j\pm \frac{1}{2}}^+\right)
 &=: g_{i \, j\pm \frac{1}{2}}^*.
\end{align}
The coefficients $c_k$ are the quadrature weights 
and  $x_{i_k}$ and $y_{j_k}$ are the quadrature points.
In an analogous way we can define the fluctuations ${\cal A}^\pm
\Delta q$ and ${\cal B}^\pm \Delta q$ along a grid cell interface.

In order to compute the flux values $f^*$ and $g^*$ at the
Gaussian quadrature points on the cell interface, we use
a Roe-type Riemann solver of the form described in Bale et al. \cite{article:BLMR2002}.
An advantage of this Riemann solver over the classical Roe solver is that
Roe averages can be replaced with simpler averages; this feature is
very helpful in the case of the MHD equations.
In any Roe-type solver, we are required to compute right and left eigenvectors
of the flux Jacobian; for the MHD system one has to be careful
about the eigenvector scalings in order to avoid singularities in the
eigenvector basis \cite{article:RoBa96}. 
For the MHD system we use the eigenvectors proposed by Barth
\cite[pages 212--214]{Barth1998}, which are based on the entropy
variables and have near optimal scaling properties for physically
admissible solution values.

We have shown here only the 2D spatial discretization. We omit the
details of the 3D spatial discretization, since the extension 
to 3D is straightforward. In 3D we use a three-dimensional limited piecewise polynomial
reconstruction of the conserved quantities and the
fluxes at grid cell interfaces are computed
by integration over a rectangular area (i.e., a surface integral instead of a line
integral as in 2D).

\subsection{Mapped grids}\label{section:3dmhd-mapped}
The spatial discretization can also be extended to two-dimensional logically
rectangular mapped grids and three-dimensional hexahedral grids. 
The vertices of each mapped grid cell in physical space, $(x,y,z)$,
are obtained by mapping the vertices
from a Cartesian mesh in computational space , $(\xi, \eta, \zeta)$ (see \cite{article:CHL08}).

In three space dimensions, each grid cell can be represented by a
trilinear map (also called a ruled cell) \cite{article:CL03,article:Us01} of the form
$$
\X(\xi,\eta,\zeta) = c_{000} + c_{100} \xi + c_{010}\eta + c_{001}
\zeta + c_{110}\xi\eta + c_{101}\xi \zeta + c_{011} \eta\zeta + c_{111}\xi\eta\zeta,
$$ 
where $\X = (x,y,z)$, $0 \le \xi, \eta, \zeta \le 1$, and $c_{000}, \ldots,
c_{111} \in \mathbb{R}^3$ are vector coefficients. 
The coefficients are determined from the
vertices of the hexahedral cell.
This representation of the grid cell can be used to
compute the volume of each grid cell, as well as face areas of all the
faces, and the unit normal vectors at each face. 
The two-dimensional case is obtained by setting
$\zeta = 0$.

Consider a multidimensional hyperbolic
system of the general form
$$
q_{,t} + \nabla \cdot {\bf F} = 0,
$$
where the columns of the matrix ${\bf F}= [\, f | g | h\, ]$ are the flux functions in
the $x$, $y$ and $z$ direction.
We consider the semi-discrete form of the finite volume method for a
three-dimensional grid cell that we denote $C_{ijk}$. The volume of
the grid cell in physical space is denoted by $|C_{ijk}|$. This
volume can be computed using the Jacobian determinant of the trilinear
map 
\begin{equation}
|C_{ijk}| = \int_0^1 \int_0^1 \int_0^1 \Big{|}\frac{\partial
  \X}{\partial(\xi,\eta,\zeta)}\Big{|} \, d\xi\, d\eta \, d\zeta.
\end{equation}
Formulas to evaluate this expression are given in \cite{article:Us01}.

The change of the cell average of the conserved quantity in grid cell
$(i,j,k)$ is described by
\begin{equation}\label{eqn:hexcell}
\begin{split}
Q_{ijk}'(t) & = - \frac{1}{|C_{ijk}|} \iiint_{C_{ijk}}
\nabla \cdot {\bf F} \, dV
 = - \frac{1}{|C_{ijk}|} \varoiint_{\partial C_{ijk}} {\bf F} \cdot \nu \, dA,
\end{split}
\end{equation}
where $\nu$ is the outward pointing normal vector along the outer
boundary of the grid cell. 

At each face of a three-dimensional grid cell, one of the variables
$\xi, \eta, \zeta$ is either zero or one and the face is
represented by a ruled surface.
Consider for example the grid cell face that corresponds to $\zeta
= 0$ in computational space (this is the interface with the index
$(i,j,k-\frac{1}{2})$). In physical space, this interface is described by the bilinear map
$$
\X(\xi,\eta) = c_{00} + c_{10} \xi + c_{01} \eta + c_{11} \xi \eta,
$$ 
which maps a square in computational space to a ruled surface embedded
in $\mathbb{R}^3$.
In order to further describe
this grid cell interface we define vectors
$$
\t_{(1)} =\frac{\partial \X}{\partial \xi} = c_{10} + c_{11} \eta, \quad \t_{(2)} =
\frac{\partial \X}{\partial \eta} = c_{01} +c_{11} \xi,
$$
which are tangent vectors to coordinate lines. 
The surface metric tensor $(a_{ij})_{i,j=1,2}$ is defined
as 
\begin{equation}
\label{eqn:aij}
a_{ij} = \t_{(i)} \cdot \t_{(j)}, \quad i,j, = 1,2,
\end{equation}
and $a=a_{11} a_{22} - a_{12} a_{21}$ denotes the determinant of the
metric tensor.
The unit normal vector to the grid
cell interface can be computed using
\begin{equation}
\label{eqn:unit_normal}
\n(\xi,\eta)  = \frac{\t_{(1)} \times \t_{(2)}}{\| \t_{(1)} \times \t_{(2)}\|}.
\end{equation}
Note that with this definition, $\n_{i-\frac{1}{2},j,k}$ is an
outward pointing normal vector for grid cell $(i-1,j,k)$ and an inward
pointing normal vector for cell $(i,j,k)$, compare with the signs of
the  terms in equation (\ref{eqn:mapped_mhd}).

The area element $dA$ on the grid cell interface transforms according to
\begin{equation}
\label{eqn:dA}
dA = \sqrt{a} \, d\xi \, d\eta = \| \t_{(1)} \times \t_{(2)} \| \, d\xi
\, d\eta.
\end{equation}
Note that the surface normal vector $\n$ and the determinant $a$ are functions of
$\xi$ and $\eta$. Analogously we can express the area element and a
normal vector at all other grid cell interfaces.
Using this, we can express (\ref{eqn:hexcell}) in computational space by
\begin{equation} \label{eqn:mapped_mhd}
\begin{split}
Q_{ijk}'(t) = \frac{-1}{|C_{ijk}|} \Bigg{[} & \int_0^1 \int_0^1 \left\{ \left( {\bf F}\cdot
  \n(\eta,\zeta) \sqrt{a(\eta,\zeta)} \right)_{i+\frac{1}{2} \, jk} - \left({\bf F} \cdot \n(\eta,\zeta)\sqrt{a(\eta,\zeta)} \right)_{i-\frac{1}{2} \, jk}\right\} \, d\eta \, d\zeta \\
 + &\int_0^1 \int_0^1 \left\{ \left({\bf F} \cdot \n(\xi,\zeta) \sqrt{a(\xi,\zeta )} \right)_{i \, j+\frac{1}{2} \, k} - \left({\bf F}
\cdot \n(\xi,\zeta)\sqrt{a(\xi,\zeta )}\right)_{i \, j-\frac{1}{2} \, k} \right\}  \, d\xi \, d\zeta \\
+ &\int_0^1 \int_0^1 \left\{ \left( {\bf F} \cdot \n(\xi,\eta)\sqrt{a(\xi,\eta)} \right)_{ij \, k+\frac{1}{2}}-
  \left({\bf F}\cdot\n(\xi,\eta)\sqrt{a(\xi,\eta)}\right)_{ij \, k-\frac{1}{2}}\right\} \, d\xi \, d\eta \Bigg{]}.
\end{split}
\end{equation}
We integrate over grid cell interfaces in
computational space using two-dimensional quadrature rules. 
The flux computation normal to the interface is again based on an
approximative Riemann solver using the
eigenvector decomposition of Barth.

\section{Spatial discretization of the non-conservative magnetic vector potential equation} 
\label{section:potential_space}
In this section we discuss the spatial discretization that is used to approximate
the evolution equation for the magnetic potential. 
For the 2D ideal MHD equations the relevant scalar evolution equation
for the magnetic potential is \eqref{eqn:A3_2D}, while in 3D the relevant weakly
hyperbolic system is \eqref{eqn:system_compact}--\eqref{eqn:N1N2N3}.
Just as in the previous section, we introduce the basic features of the
numerical scheme for a one-dimensional problem.

\subsection{One-dimensional weakly hyperbolic systems}\label{section:1d-potential}
We consider an equation of the general form
\begin{equation} \label{eqn:weak-hyp}
q_{,t} + A(x) \, q_{,x} = 0,
\end{equation}
with $q:\mathbb{R}^+ \times \mathbb{R} \rightarrow \mathbb{R}^m$ and $A(x)
\in \mathbb{R}^{m\times   m}$. We assume that system (\ref{eqn:weak-hyp}) is 
weakly hyperbolic, which means that  matrix $A(x)$ has real
eigenvalues but not always a complete set of right eigenvectors. As in
the classical hyperbolic case, we wish to
construct a method of the form (\ref{method1}). However, due to the
weak hyperbolicity, the fluctuations
${\cal A}^\pm \Delta q$ cannot be approximated via a wave
decomposition as in \cite{article:BLMR2002,sharpclaw,article:LeVeque97}.

Let $\tilde{q}$ be a piecewise polynomial reconstruction of the
function $q$ as described in (\ref{eqn:tildeq}). Furthermore, 
let $\tilde{A}(x)$ denote a piecewise polynomial approximation of the
matrix $A(x)$:
\begin{equation}
\begin{split}
\tilde{A}(x) & = \tilde{A}_i(x) \quad \text{for} \quad x \in \left(x_{i-\frac{1}{2}},x_{i+\frac{1}{2}} \right), \\
\tilde{A}_i(x) & = A(x) + {\cal O}\left(\Delta x^{p+1}\right).
\end{split}
\end{equation}
The values $q_{i-\frac{1}{2}}^{\pm}$ are the reconstructed values at the grid cell
interface as defined in (\ref{qpm}). Analogously we define 
\begin{equation}
A_{i-\frac{1}{2}}^- := \lim_{\varepsilon \rightarrow 0} \, \tilde{A}_{i-1}\left(x_{i-\frac{1}{2}} - \varepsilon \right) \quad \text{and}
\quad A_{i-\frac{1}{2}}^+ := \lim_{\varepsilon \rightarrow 0} \, \tilde{A}_i\left(x_{i-\frac{1}{2}} + \varepsilon \right).
\end{equation}
Using these reconstructions, we derive a
method of the form  
\begin{equation*}
 Q'_i(t) = -\frac{1}{\Delta x} \left( {\cal
    A}^- \Delta q_{i+\frac{1}{2}} + {\cal A}^+ \Delta
  q_{i-\frac{1}{2}} + {\cal A}\Delta q_i \right),
\end{equation*}
where $Q_i(t)$ is the cell average of the quantity $q$ and 
\begin{eqnarray}
\label{adq}
\Am \Delta q_i & \approx & \lim_{\varepsilon \rightarrow 0}
\int_{x_{i-\frac{1}{2}}+\varepsilon}^{x_{i+\frac{1}{2}}-\varepsilon} A(x) \,
q_x \, dx, \\ 
\label{apdq}
 \Am^+ \Delta q_{i-1/2} & \approx & \lim_{\varepsilon
 \rightarrow 0} \int_{x_{i-1/2}}^{x_{i-1/2} +
 \varepsilon} {A}(x) \, {q}_x \, dx, \\ 
 \label{amdq}
 \Am^- \Delta q_{i+ 1/2} & \approx & \lim_{\varepsilon
 \rightarrow 0} \int_{x_{i+1/2}- \varepsilon}^{x_{i+1/2}} {A}(x) \, {q}_x \,
 dx.
\end{eqnarray}
The integral in (\ref{adq}) is approximated by replacing $A(x)$ and
$q$ by the approximating polynomials $\tilde{A}_i$ and
$\tilde{q}_i$, i.e.\
$$
\Am \Delta q_i := \int_{x_{i-\frac{1}{2}}}^{x_{i+\frac{1}{2}}}
\tilde{A}_i(x)  \, \tilde{q}_{i,x} \, dx.
$$
This integral can then be computed using Gaussian quadrature of the
appropriate order.

In the weakly hyperbolic case, the fluctuations $\Am^\pm \Delta q$ 
cannot be computed with the wave propagation algorithm, since the
wave propagation method requires a complete set of eigenvectors.
In order to compute these fluctuations, we introduce a regularization
of $q$ at each grid cell interface of the  form
\begin{equation*}
q^\varepsilon_{i-\frac{1}{2}}(t,x) = \left\{ \begin{array}{c c c}
q_{i-\frac{1}{2}}^-(t) & : & x \le x_{i-\frac{1}{2}} - \varepsilon, \\
\Psi_{i-\frac{1}{2}} \left( \frac{ x-x_{i-\frac{1}{2}}+\varepsilon}{2 \varepsilon},t
\right) & : & x
\in \left( x_{i-\frac{1}{2}}-\varepsilon, x_{i-\frac{1}{2}}+
\epsilon \right), \\
q_{i-\frac{1}{2}}^+(t) & : & x\ge x_{i-\frac{1}{2}}+\varepsilon,
\end{array}\right.
\end{equation*}  
where $\Psi_{i-\frac{1}{2}}$ is a path connecting
$q_{i-\frac{1}{2}}^-(t)$ and $q_{i-\frac{1}{2}}^+(t)$. Using this
regularization, we first construct an expression for the sum 
of the left and right-going fluctuations at the grid cell interface
$x_{i-\frac{1}{2}}$, i.e., we wish to express
$$
{\cal A}^- \Delta q_{i-\frac{1}{2}} + {\cal A}^+ \Delta
q_{i-\frac{1}{2}} \approx \lim_{\varepsilon \rightarrow 0}
\int_{x_{i-\frac{1}{2}}-\varepsilon}^{x_{i-\frac{1}{2}}+\varepsilon} A(x) \,
q_x \, dx. 
$$
This integral is approximated by replacing $q$ with the regularized
function $q^\epsilon$ and by replacing the matrix $A$ with the
piecewise polynomial approximation $\tilde{A}(x)$, i.e., we set
\begin{equation}\label{int-Aqx1}
\begin{split}
{\cal A}^- \Delta q_{i-\frac{1}{2}} + {\cal A}^+ \Delta
q_{i-\frac{1}{2}} 
 & = \lim_{\varepsilon \rightarrow 0} \int_{x_{i-\frac{1}{2}}-\varepsilon}^{x_{i-\frac{1}{2}}+\epsilon}
\tilde{A}(x) \left(q_{i-\frac{1}{2}}^{\varepsilon}(t,x) \right)_{,x} \, dx.
\end{split}
\end{equation}
After the substitution $s = \frac{x-x_{i-\frac{1}{2}}+\epsilon}{2
  \epsilon}$ and by defining $\phi(s) :=
x_{i-\frac{1}{2}}-\epsilon + s(x_{i-\frac{1}{2}} + \epsilon -
(x_{i-\frac{1}{2}}-\epsilon))$, the integral in
(\ref{int-Aqx1}) can be written as
\begin{equation}
\int_{x_{i-\frac{1}{2}}-\varepsilon}^{x_{i-\frac{1}{2}}+\epsilon}
\tilde{A}(x) \left(q_{i-\frac{1}{2}}^{\varepsilon}(t,x) \right)_{,x} \, dx
=
\int_0^1 \tilde{A}(\phi(s)) \left(\Psi_{i-\frac{1}{2}}(s,t) \right)_{,s} \, ds.
\end{equation}
In order to resolve this integral,  we choose the simple straight-line
path 
\begin{equation}
\Psi_{i-\frac{1}{2}} = q_{i-\frac{1}{2}}^- + s \, \left(q_{i-\frac{1}{2}}^+ -
q_{i-\frac{1}{2}}^- \right),
\end{equation}
and obtain
\begin{equation}
\int_0^1 \tilde{A}(\phi(s)) \left(\Psi_{i-\frac{1}{2}}(s,t) \right)_{,s} \, ds
 =
\int_0^1 \tilde{A}\left(\phi(s) \right)   \, ds \, \left(q_{i-\frac{1}{2}}^+(t) - q_{i-\frac{1}{2}}^-(t) \right).
\end{equation}
The integral on the right-hand side can be split into two pieces:
\begin{equation}
\int_0^1 \tilde{A} (\phi(s)) \, ds = 
\int_0^{1/2} \tilde{A}_{i-1}(\phi(s))\, ds + \int_{1/2}^{1} \tilde{A}_i(\phi(s))\, ds,
\end{equation}
which corresponds to a division into pieces on the left and right of the
discontinuity $x_{i-\frac{1}{2}}$. We then take
the limit $\varepsilon \rightarrow 0$ and get:
\begin{equation}
\lim_{\varepsilon \rightarrow 0} \left( \int_0^{1/2} \tilde{A}_{i-1}(\phi(s))\,
ds + \int_{1/2}^{1} \tilde{A}_i(\phi(s))\, ds \right) 
 = 
\frac{1}{2} {A}_{i-\frac{1}{2}}^- +
\frac{1}{2} {A}_{i-\frac{1}{2}}^+ =: A\bigl|_{\Psi_{i-\frac{1}{2}}}.
\end{equation}
Thus we obtain
\begin{equation}
\label{eqn:nonc_fluctuation}
\Am^- \Delta q_{i-\frac{1}{2}} + \Am^+ \Delta q_{i-\frac{1}{2}}
=  A\bigl|_{\Psi_{i-\frac{1}{2}}} (q_{i-\frac{1}{2}}^+ -
q_{i-\frac{1}{2}}^-). 
\end{equation}

Following Castro et al. \cite{toro09} and using 
relationship \eqref{eqn:nonc_fluctuation}, we are now able
to define the left and right-going fluctuations.
With partial knowledge of the eigenstructure of $A\bigl|_{\Psi}$ (i.e.,\ if
we know the eigenvalues or
an estimate of the largest absolute value of the eigenvalues) we can define
\begin{eqnarray}\label{amdq-weakly}
\Am^- \Delta q_{i+1/2} = \frac{1}{2}
\Biggl[\underbrace{A\bigl|_{\Psi_{i+1/2}} \, 
- \, \alpha_{i+1/2} \, {\Id}}_{\text{generalized Rusanov flux}} \Biggr]  \Biggl({q}_{i+\frac{1}{2}}^+ -
{q}_{i+\frac{1}{2}}^-\Biggr)
\label{fluct_m_ev}
\end{eqnarray}
and
\begin{eqnarray}\label{apdq-weakly}
\Am^+ \Delta
q_{i-1/2}=\frac{1}{2} \Biggl[\underbrace{A\bigl|_{\Psi_{i-1/2}} \,
+ \, \alpha_{i-1/2} \, {\Id}}_{\text{generalized Rusanov flux}}\Biggr] \Biggl( {q}_{i-\frac{1}{2}}^+-{q}_{i-\frac{1}{2}}^-
\Biggr),
\label{fluct_p_ev}
\end{eqnarray}
where ${\Id} \in \mathbb{R}^{m\times m}$ is the identity matrix.
Here $\alpha$ is a positive number with
$$|\lambda^k| \leq \alpha, \quad \text{for} \quad {k=1,\ldots,m},$$
where $\lambda^k$ represents the $k^{\text{th}}$ eigenvalue of $A\bigl|_{\Psi}$,
see \cite{toro09}.

 Another possibility is to define the fluctuations without
using any knowledge of the eigenstructure of $A\bigl|_{\Psi}$:
\begin{eqnarray}
\Am^- \Delta q_{i+1/2} = \frac{1}{4}
\Biggl[2 A\bigl|_{\Psi_{i+\frac{1}{2}}}-\frac{\Delta x}{\Delta t} {\Id}
-\frac{\Delta t}{\Delta x} \left(A\bigl|_{\Psi_{i+\frac{1}{2}}} \right)^2 \Biggr]
\Biggl({q}_{i+\frac{1}{2}}^+ - {q}_{i+\frac{1}{2}}^-\Biggr)
\label{fluct_m}
\end{eqnarray}
and
\begin{eqnarray}
\Am^+ \Delta
q_{i-1/2}=\frac{1}{4} \Biggl[2 A\bigl|_{\Psi_{i-\frac{1}{2}}}+\frac{\Delta x}{\Delta t}
{\Id} +\frac{\Delta t}{\Delta x} \left( A\bigl|_{\Psi_{i-\frac{1}{2}}} \right)^2 \Biggr]
\Biggl( {q}_{i-\frac{1}{2}}^+-{q}_{i-\frac{1}{2}}^- \Biggr).
\label{fluct_p}
\end{eqnarray}
These fluctuations are derived from the generalized FORCE scheme
\cite{dumb09_1D,toro09}, which is a convex combination of the generalized
Lax-Friedrichs and the generalized Lax-Wendroff scheme.

\subsection{Multidimensional Cartesian grids}
We now consider the approximation of a weakly hyperbolic system of the form 
\begin{equation}
\label{eqn:2d-weakhyp}
q_{,t} + A(x,y) \, q_{,x} + B(x,y) \, q_{,y} = 0,
\end{equation}
with $q: \mathbb{R}^+ \times \mathbb{R}^2 \rightarrow \mathbb{R}^m$,
$A(x,y), \ B(x,y) \in \mathbb{R}^{m\times m}$.
In the semi-discrete  case, the method has the form
\begin{equation}
\begin{split}
Q_{ij}'(t)  = & -\frac{1}{\Delta x} \left( {\cal A}^+ \Delta
  q_{i-\frac{1}{2} \, j} + {\cal A}^- \Delta q_{i+\frac{1}{2} \, j} + {\cal
    A} \Delta q_{ij} \right)\\
&  - \frac{1}{\Delta y} \left( {\cal B}^+ \Delta q_{i,j-\frac{1}{2}}
  + {\cal B}^- \Delta q_{i \, j+\frac{1}{2}} + {\cal B} \Delta q_{ij} \right). 
\end{split}
\end{equation}
As in the case of hyperbolic systems (see Section \ref{multiD_hyp}), we
compute a multidimensional piecewise polynomial reconstruction
$\tilde{q}$ of $q$ using a least squares approach.
Furthermore, $\tilde{A}(x,y)$ and $\tilde{B}(x,y)$ are piecewise
polynomial approximations of the matrix valued functions. Using these
reconstructions we define
\begin{equation}
\begin{split}
{\Am} \Delta q_{ij} & := \frac{1}{\Delta y}
\int_{y_{j-\frac{1}{2}}}^{y_{j+\frac{1}{2}}}
\int_{x_{i-\frac{1}{2}}}^{x_{i+\frac{1}{2}}} \tilde{A}_{ij}(x,y)
\, \tilde{q}_{ij, x} \, dx \, dy, \\
{\cal B} \Delta q_{ij} & := \frac{1}{\Delta x} \int_{y_{j-\frac{1}{2}}}^{y_{j+\frac{1}{2}}}
\int_{x_{i-\frac{1}{2}}}^{x_{i+\frac{1}{2}}} \tilde{B}_{ij} (x,y)
\, \tilde{q}_{ij, y} \, dx \, dy.
\end{split}
\end{equation} 
The fluctuations ${\cal A}^\pm$ and ${\cal B}^\pm$ are defined in
analogy to (\ref{amdq-weakly}) and (\ref{apdq-weakly}) with the only
difference that we integrate over a grid cell interface and then
compute the average value.
For example, we compute
\begin{equation}
\Am^- \Delta q_{i+\frac{1}{2} \, j} = \frac{1}{\Delta y}
\int_{y_{j-\frac{1}{2}}}^{y_{j+\frac{1}{2}}} \frac{1}{2} 
\Biggl[ A \bigl|_{\Psi_{i+1/2 \, j}} (y) \, 
- \, \alpha_{i+\half}(y) \, {\Id} \Biggr]  \Biggl({q}_{i+\half \, j}^+(y) -
{q}_{i+\frac{1}{2} \, j}^- (y) \Biggr) \, dy.
\end{equation}  
We evaluate this integral using Gaussian quadrature formulas. This means
that we need to compute the matrix $A  \bigl|_{\Psi_{i+\half \, j}}$ at the
nodes of the quadrature formula. The other fluctuations are computed
in an analogous way.

\subsection{Mapped grids}
We now discuss the discretization of a weakly hyperbolic system of the
form (\ref{eqn:2d-weakhyp}) on a logically rectangular mapped grid.
For simplicity we restrict our considerations to the two-dimensional
case. The extension to the three-dimensional case is a combination of
the concepts discussed in this section and those of Section
\ref{section:3dmhd-mapped}.

Let $C_{ij}$ denote a grid cell in physical space and $|C_{ij}|$ the
area of the grid cell. Assume that the grid cell in physical space is
obtained by a mapping of the unit square in computational space. The
change of the cell average of the quantity $q$ in grid cell $C_{ij}$
is described by the ODE: 
\begin{equation}
\begin{split}
Q_{ij}'(t)  = - \frac{1}{|C_{ij}|} \Big( & 
\ell_{i-\frac{1}{2} \, j} \, {\Am^+ \Delta q_{i-\frac{1}{2} \, j}}  
+ {\tilde{A}_{ij}  \, \tilde{q}_{ij, x}} 
+ \ell_{i+\frac{1}{2} \, j}  \,  {\Am^- \Delta  q_{i+\frac{1}{2} \, j}} \\
& + \ell_{i \, j-\frac{1}{2}}  \,  {\Bm^+ \Delta
  q_{i,j-\frac{1}{2}}}  
+ {\tilde{B}_{ij}  \, \tilde{q}_{ij, y}} + \ell_{i \, j+\frac{1}{2}}  \,  {\Bm^- \Delta
    q_{i,j+\frac{1}{2}}} \Big), 
\end{split}
\end{equation}
where 
$\ell_{i+\frac{1}{2} \, j}$ and $\ell_{i \, j+\frac{1}{2}}$ denote the
length of the grid cell interfaces in physical space.
The integrals over the grid cell are defined as
\begin{equation}
\begin{split}
{\tilde{A}_{ij}  \, \tilde{q}_{ij, x}} & = \iint_{C_{ij}} \tilde{A}_{ij}(x,y)
 \, \tilde{q}_{ij, x} (x,y) \, dx  \, dy\\
& =   \int_0^1 \int_0^1 \sqrt{a_{ij}(\xi,\eta)}  \,  \tilde{A}_{ij}(x(\xi,\eta),y(\xi,\eta))  \, 
\tilde{q}_{ij, x} (x(\xi,\eta),y(\xi,\eta))\, d\xi  \, d\eta, \\
{\tilde{B}_{ij}  \, \tilde{q}_{ij, y}} & =  \int_0^1 \int_0^1 \sqrt{a_{ij}(\xi,\eta)}  \, 
\tilde{B}_{ij}(x(\xi, \eta),y(\xi,\eta))  \, \tilde{q}_{ij, y}(x(\xi,\eta),y(\xi,\eta)) \, d\xi  \, d\eta,
\end{split}
\end{equation}
where we use piecewise polynomial reconstructed functions $\tilde{q}$
and a piecewise polynomial representations of the matrix valued
functions.
For a planar two-dimensional mapped grid cell, the 
area element $\sqrt{a}$ is defined by
$$
\sqrt{a} = \Bigl| (c_{10}^1 + c_{11}^1 \eta )(c_{01}^2 + c_{11}^2 \xi ) -
(c_{01}^1 + c_{11}^1 \xi ) (c_{10}^2 + c_{11}^2 \eta ) \Bigr|.
$$
The coefficients $c_{10}, \ldots, c_{11} \in \mathbb{R}^2$ are special
cases of those detailed in \S \ref{section:3dmhd-mapped}, and can be obtained from
those coefficients by setting $\zeta$ and the third
component of $\X$ equal to zero.

In order to define the fluctuations ${\Am^\pm \Delta q}$ and ${\Bm^\pm
  \Delta q}$, we introduce $\n_{i\pm\frac{1}{2} \, j}$ and
$\n_{i \, j\pm\frac{1}{2}}$, the normal vectors at the interfaces
$(i\pm\frac{1}{2} \, j)$ and $(i \, j\pm\frac{1}{2})$. 
We define $A \bigl|_{\Psi}$ by replacing the matrix
$A$ in the one-dimensional formula of Section
\ref{section:1d-potential} by $\hat{A} = n^1 A + n^2 B$. Now we
integrate the fluctuations over each grid cell interface using
again the piecewise polynomial reconstructed values.  
For example we compute
\begin{equation}
{\Am^- \Delta q_{i+\frac{1}{2} \, j}} = 
  \int_0^1 \frac{1}{2} 
\Biggl[ \hat{A} \bigl|_{\Psi_{i+\half \, j}} (\eta) \, 
- \, \alpha_{i+\half}(\eta) \, {\Id} \Biggr]  \Biggl({q}_{i+\frac{1}{2} \, j}^+(\eta) -
{q}_{i+\half \, j}^- (\eta) \Biggr) \, d\eta.
\end{equation}
All the other fluctuations are computed analogously.

\subsection{Limiting with respect to the derivative}
\label{sec:special_limiters}
As first pointed out by Rossmanith in \cite{article:Ro04b}, a special limiting is
needed for the update of the magnetic potential in order to avoid
unphysical oscillation in derivatives of the potential and thus in the
magnetic field. He proposed an extension of the limiting mechanism
used in the wave propagation algorithm. 

Here we describe a limiting strategy that can be used in combination
with the update of the potential that is proposed in this paper.
The strategy of limiting with respect to the derivative will be
described for the 1D advection equation
\begin{eqnarray}\label{advection}
q_{,t} + u q_{,x} = 0,
\label{1d_advection_eqn}
\end{eqnarray} 
where $q:\mathbb{R}\times \mathbb{R}^+ \rightarrow \mathbb{R}$ is 
continous but not necessarily continously differentiable. 
The advection speed $u\in \mathbb{R}$ is assumed to be constant and we
restrict to the case $u>0$.
We use the method introduced in Section \ref{section:1d-potential} 
together with the third order accurate SSP-RK method to
update the advected quantity  $q$.

In order to limit the solution we add numerical viscosity\footnote{In the context of
the evolution of the magnetic potential in ideal MHD, the term {\it numerical viscosity}
really refers to {\it numerical resistivity}.} to the
problem, i.e.\ instead of (\ref{advection}) we approximate an  
advection diffusion equation of the form
\begin{equation}
q_{,t} + u q_{,x} = \varepsilon(x) \, q_{,x,x}.
\end{equation}
This scalar advection-diffusion equation can also be written 
as a system of first order equations
\begin{align}
q_{,t} + u q_{,x}  &=  \varepsilon(x) \, d_{,x}, \\
d - q_{,x} &= 0.
\end{align} 
Usually $\varepsilon(x)$ is chosen such that a thin layer is added near discontinuities
in the solution itself while keeping the high order
reconstruction away from the discontinuity (e.g., see \cite{persson}).
We will change the definition of $\varepsilon$ in such a way that
additional viscosity is added near jumps in the derivative of the solution.

The artificial viscosity must satisfy the following requirements: 
\begin{enumerate}
  \item{It should be small
enough to satisfy the stability constraint
\begin{eqnarray*}
\frac{\varepsilon(x) \Delta t}{(\Delta x)^2} \leq \frac{1}{2}
\end{eqnarray*}
of our explicit time-stepping method.}
\item{It should be large enough to avoid spurious oscillations in the derivative due to Gibbs
phenomenon.}
\item{The artificial viscosity should not degrade
the acurracy of the scheme for smooth solution structures.}
\end{enumerate}
To be able to handle all these requirements we define the viscosity in the
following way:
\begin{eqnarray}
\varepsilon(x) = \eta \alpha,
\end{eqnarray}
with 
\begin{eqnarray}
\eta = 0.2\frac{(\Delta x)^2}{\Delta t},
\label{viscosity}
\end{eqnarray}
and 
\begin{eqnarray}
\alpha = \begin{cases} \frac{1}{2} \bigl[ 1+\text{sin}\left(\pi \, \Delta S - \frac{\pi}{2} \right) \bigr] & S >
\sigma_{ii}, \\ 0 & S \leq \sigma_{ii}, \end{cases}
\label{smoothness_ind}
\end{eqnarray}
where 
\begin{equation}
\label{eqn:artvisc_params}
\begin{split}
S = \text{max}(\sigma_{i \, i-1},\sigma_{i \, i+1}), \quad \Delta S = \bigl| S- \sigma_{ii} \bigr|, \\
\sigma_{ik} = \frac{\lambda_{ik}}{((\Delta x)^4 + \Sigma_{ik})^e}, \quad
\Sigma_{ik} = (\tilde{q}_{,x,x}(k)(\Delta x)^2)^2, \quad k \in {(i-1,i,i+1)}.
\end{split}
\end{equation}
The parameters are set to $\lambda_{i \, i-1}=\lambda_{i \, i+1}=1, \lambda_{ii}=1000$ and $e=4$ in the test below. 
Note that we use a double index on the parameters $\lambda$ and $\sigma$; the logic
behind this notation is that the first index denotes the current cell that is being considered and the
second index denotes the cell to which the current cell is being compared.
$\eta$ is the maximum amount of viscosity that could be added and $\alpha$ is a
smoothness indicator, with values between 0 and 1, such that $\alpha$ approaches $1$ on very steep gradients in the derivative (i.e., steep second
derivative), and $\alpha$ is $0$ in smooth regions of the derivative.
The idea behind this smoothness indicator is that we do not take any viscosity
(large linear weight $\lambda_{ii}$), unless the second derivative in one of the neighbouring cells is much higher than in cell $i$, 
in which case we will add some viscosity depending on the difference of the smoothness measure $\sigma_{ii}$ and the maximum of the smoothness measures in the neighouring cells ($\sigma_{i \, i-1}, \sigma_{i \, i+1})$.
This means that the parameter $\lambda_{ii}$ controls how much the second derivatives have to differ, before adding any viscosity. We found that $\lambda_{ii} \in [10^1,10^3]$ seem to give good results.
From the form of the viscosity, it is clear that we do not obtain a further time
step restriction due to the viscosity limiting. 
Taking the time step restriction into account, one could also choose $\eta= \mathcal{O}(\Delta x)$ empirically.

We now consider the 1D advection example that was first introduced in \cite{article:Ro04b}.
Consider \eqref{advection} with $u=1$ on the interval $0 \leq x \leq 1$
and double periodic boundary conditions. The initial data is the
following piecewise linear function:
\begin{eqnarray*}
q(0,x) = \begin{cases} 0 &\text{if} \quad x \leq0.25, \\
(x-0.25)/0.075 & \text{if} \quad 0.25 \leq x \leq 0.4,\\
2 & \text{if} \quad 0.4 \leq x \leq 0.6,\\
(0.75 - x)/0.075 & \text{if} \quad 0.6 \leq x \leq 0.75,\\
0 & \text{if} \quad 0.75 \leq x.
\end{cases}
\end{eqnarray*}
Note that $q(0,x)$ is continuous, but its first derivative is discontinuous.
The challenge is to control oscillations in both $q(t,x)$ and $q_{,x}(t,x)$.

\begin{figure}
\begin{center}
(a)\includegraphics[width=0.46\textwidth]{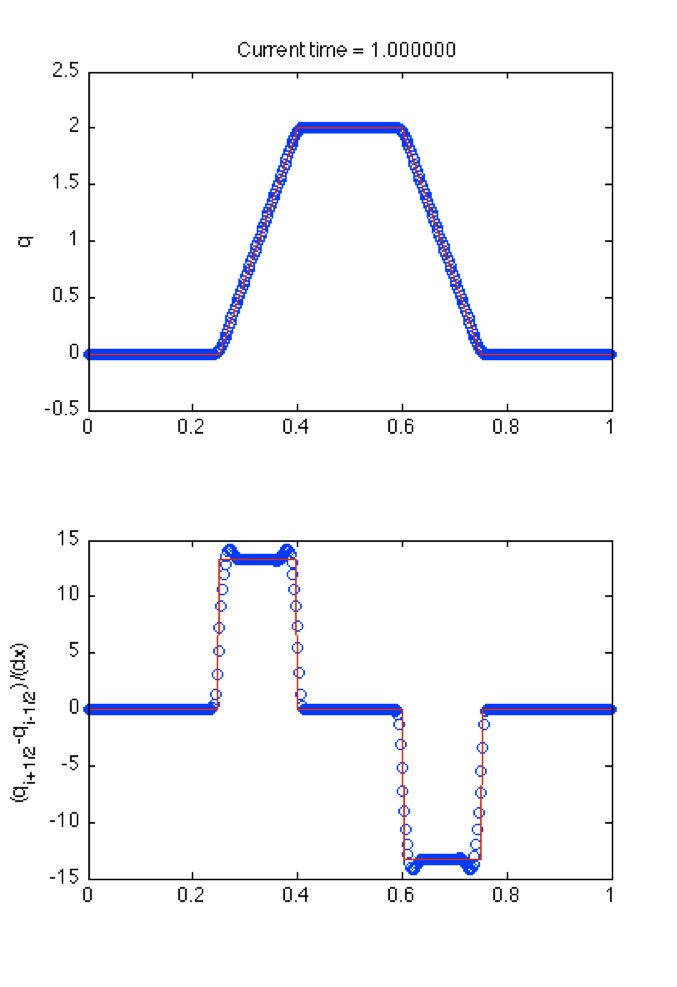}\hfill
(b)\includegraphics[width=0.46\textwidth]{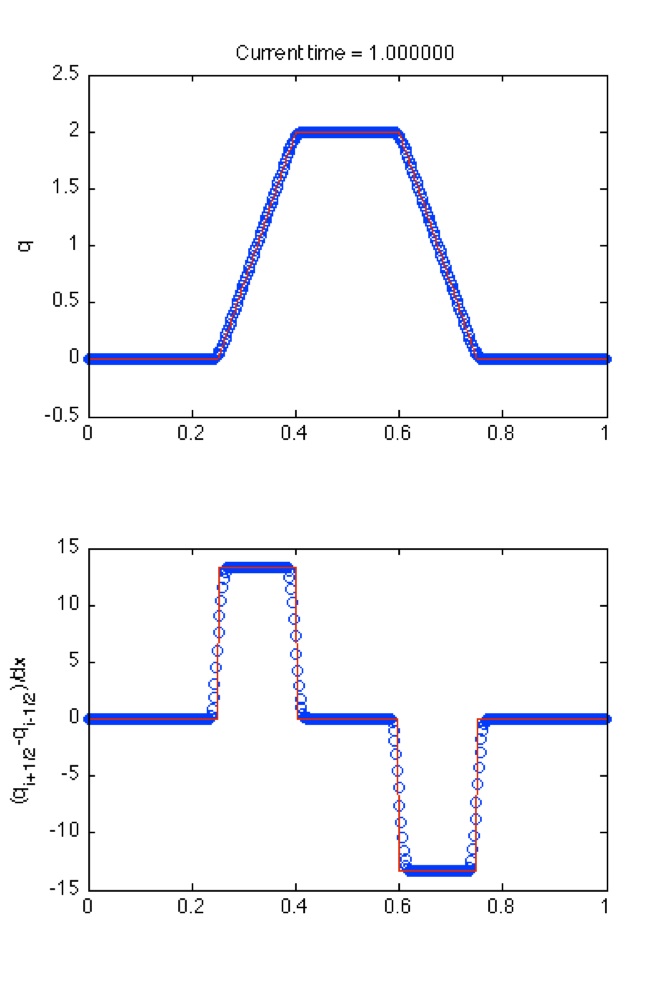}
\caption{The solution to the 1D advection equation and its derivative on a
periodic domain after one revolution using the $3^{rd}$ order
algorithm. Panel (a) shows how standard WENO limiting performs on the solution,
and Panel (b) shows how the limiter proposed in this work performs.
\label{fig:adv1d_limi}}
\end{center}
\end{figure}

\noindent
We compare the solution and its derivative as computed by two
different limiters: (a) a standard WENO limiter that is essentially
non-oscillatory in $q(t,x)$, but not necessarily in $q_{,x}(t,x)$, and
(b) the proposed limiter as described above. 
For further
details on the WENO limiting see Canestrelli et al.\ \cite{dumb09_1D},
Titarev et al.\ \cite{ti_10} and Tsoutsanis et al.\ \cite{ts_11}. 
Third order accurate average values of the derivatives $q_{,x}(t,x)$ are
computed using
\begin{equation}
\begin{split}
\frac{1}{\Delta x} \int_{x_{i-\frac{1}{2}}}^{x_{i+\frac{1}{2}}}
q_{,x}(t,x) \, dx & = \frac{1}{\Delta x} \left[ q\left(t,x_{i+\frac{1}{2}}\right) -
  q\left(t,x_{i-\frac{1}{2}}\right) \right] \\
& \approx \frac{1}{\Delta x} \left[ \frac{1}{2}\left(q_{i+\frac{1}{2}}^-(t) +
  q_{i+\frac{1}{2}}^+(t)\right)
- \frac{1}{2}\left(q_{i-\frac{1}{2}}^-(t) +
 q_{i-\frac{1}{2}}^+(t)\right) \right].
\end{split}
\end{equation}
We use a CFL number of $0.7$ in both simulations.
The solution and its derivative at time $t=1$ (i.e., after one revolution) are shown
in Figure \ref{fig:adv1d_limi}, where Panel (a) is the WENO limiter
and Panel (b) is the proposed limiter.
The approximation of $q(1,x)$ looks very similar
in both cases.
However, there are obvious differences in the computed results for
$q_{,x}(t,x)$; the proposed limiter avoids spurious oscillations in $q_{,x}(t,x)$. 
In comparison with the TVD scheme of \cite{article:Ro04b} the limiting introduced here
(which can also be used in a one-step scheme) seems to be slightly more
diffusive. This however is a general problem if a method of lines approach is used instead
of a Lax-Wendroff type method.

In the multidimensional case, we compute the smoothness indicator $S$
by taking all neighboring grid cells into account that share a face, an edge or a
corner point with the considered grid cell. 
For the computation of  $\Sigma$, the second derivative of $\tilde{q}$
is replaced by the Laplacian and $(\Delta x)^2$ is replaced by 
$(|C_{ij}|)$ in the two-dimensional case, and  $(|C_{ijk}|)$ in
the three-dimensional case.
As characteristic length in the computation of $\eta$ and $\sigma$ ($\Delta x$ in 1D), 
we use the length of the smallest edge of the considered grid cell in the multidimensional case.
Denoting the characteristic length by $\Delta s$, we use $\eta = 0.5 \Delta s$ in all MHD computations below and 
$\lambda_{i \, i-1} = 1$, $\lambda_{i \, i+1} = 1$, and $\lambda_{ii}=10^3$ (see formulas in \eqref{eqn:artvisc_params}).
  
\section{The discretization of $\nabla \times \Av$}
\label{sec:curlA}
During each stage of the constrained transport algorithm described in Section
\ref{section:temp}, we obtain the corrected  magnetic field from taking the curl of
the magnetic potential.
In our previous work on CT schemes \cite{article:HeRoTa10,article:Ro04b}, the cell average
values of the magnetic field were computed using centered finite
difference approximations. In order for the scheme developed in this work
to be third order accurate both on Cartesian and mapped grids,
we must generalize the central finite difference approach. 
We describe this generalization in this section.

We wish to compute a discrete version of the curl of the magnetic
vector potential:
\begin{equation}
\B = \nabla \times \Av = \left[  A^3_{,y} - A^2_{,z}, \quad
 A^1_{,z}-   A^3_{,x}, \quad  A^2_{,x} - A^1_{,y} \right]^T.
\end{equation}
In particular, we wish to compute a third order accurate estimate
of the cell average of $\B$. For example, using the
divergence theorem, the cell average of the first component of the magnetic field, $B^1$,
in grid cell $(i,j,k)$ can be expressed in the form
\begin{eqnarray}
B^1_{ijk} & = & \frac{1}{|C_{ijk}|} \iiint_{C_{ijk}}
\left(A^3_{,y} - A^2_{,z} \right) \, dV \nonumber \\
& = & \frac{1}{|C_{ijk}|} \iiint_{C_{ijk}} \nabla \cdot
\left( \begin{array}{c}
0\\A^3\\-A^2\end{array}\right) \, dV \nonumber\\
& = & \frac{1}{|C_{ijk}|} \varoiint_{\partial C_{ijk}} \left( \begin{array}{c}
0\\A^3\\-A^2\end{array}\right) \cdot \nu \, dA \nonumber\\
& = & \frac{1}{|C_{ijk}|} \varoiint_{\partial C_{ijk}} \left( A^3 \nu^2
  - A^2 \nu^3\right) \, dA, \label{eqn:approx-B1}
\end{eqnarray}
where $\nu = (\nu^1,\nu^2,\nu^3) \in \mathbb{R}^3$ is the outward pointing unit normal vector along the boundary
$\partial C_{ijk}$ of the considered grid cell.
The integrals over all the faces of a grid cell can be evaluated in
comptational space as discussed in Section
\ref{section:3dmhd-mapped}. 
Using the definition of the surface normal vector $\n$ and the
determinant $a$ of the mapping as defined for the grid cell interfaces
in Section
\ref{section:3dmhd-mapped},
we obtain
\begin{equation}\label{eqn:approx-B1-2}
\begin{split}
B_{ijk}^1 
& =  \frac{1}{|C_{ijk}|} \sum_{\pm = +,-} \Big\{ \pm
\int_0^1 \int_0^1   \left( \left(A^3 n^2(\eta,\zeta) - A^2
  n^3(\eta,\zeta)
\right)\sqrt{a(\eta,\zeta)}\right)_{i \pm \frac{1}{2},j,k}  \, d\eta
d\zeta\\
& \qquad
\pm \int_0^1 \int_0^1  \left( \left(A^3 n^2(\xi,\zeta) - A^2
  n^3(\xi,\zeta)
\right)\sqrt{a(\xi,\zeta)}\right)_{i,j \pm \frac{1}{2},k}  \, d\xi
d\zeta\\
& \qquad
\pm \int_0^1 \int_0^1  \left( \left(A^3 n^2(\xi,\eta) - A^2
  n^3(\xi,\eta)
\right)\sqrt{a(\xi,\eta)}\right)_{i,j,k \pm \frac{1}{2}}  \, d\xi
d\eta \Big\}.
\end{split}
\end{equation}

In order to evaluate the surface integral in \eqref{eqn:approx-B1-2}, 
we first replace the integrals with a Gaussian quadrature rule of the
appropriate degree of precision, then we reconstruct $A^2$ and $A^3$ on the
grid cell boundaries at each of the quadrature points 
using the piecewise polynomial reconstruction of $\Av$. 
To be more precise, at each of these quadrature points we take the value of $\Av$ to
be the average of the reconstructed values from the two grid cells that share a face. 
We note that this discretization leads to a conservative update of the magnetic field. 
For a smooth magnetic potential the average of the reconstructed values agrees with
the correct interface value of the magnetic potential up to the order
used in the reconstruction.  
The computation of the cell averages of $B^2$ and $B^3$ can be done in
an analogous way using
\begin{equation*}
\begin{split}
B_{ijk}^2 & = \frac{1}{|C_{ijk}|} \varoiint_{\partial C_{ijk}} \left(
  A^1 \nu^3 - A^3 \nu^1 \right) \, dA \\
B_{ijk}^3 & =  \frac{1}{|C_{ijk}|} \varoiint_{\partial C_{ijk}} \left(
A^2 \nu^1 - A^1 \nu^2 \right) \, dA.
\end{split}
\end{equation*}

\begin{claim}
The constrained transport method as described in this work locally conserves (and therefore
also globally conserves) the magnetic field, $\B$, from one Runge-Kutta stage to the next.
\end{claim}
\begin{proof}
Without loss of generality consider a single Euler step (i.e., a single stage in the SSP time-stepping scheme)
on the magnetic field:
\begin{equation}
\begin{split}
\B^{n+1} &= \nabla \times \Av^{n+1} = \nabla \times \left( \Av^{n} + \Delta t \, \E^{n} \right) \\ 
	&= \B^{n} + \Delta t \, \nabla \times \E^{n} = 
	\B^{n} + \Delta t \, \nabla \cdot \left( \left[ \epsilon_{ijk} E^k \right]^n \right).
\end{split}
\end{equation}
The divergence of the tensor $[\epsilon_{ijk} E^k]$ is computed via
surface integrals of the form \eqref{eqn:approx-B1-2}, where $\Av$ is replaced with $\E$.
The values of $\E = \left( \nabla \times \Av \right) \times \u$ are computed from constrained transport as described in Section \ref{section:potential_space}
and then averaged onto appropriate Gaussian quadrature points on each face.
In particular, note that the same values of  $\E$ are used to update
the cell average of $\nabla \cdot \left([\epsilon_{ijk} E^k]\right)$ on either side of the face. Therefore, we are guaranteed
that the discrete magnetic field is locally conserved over each Euler time step.
\end{proof}

\section{Numerical Results}
\label{section:test-computations}
Several numerical examples on Cartesian and mapped grids are
shown in this section. These examples are used to both
verify the third-order accuracy in space and time, as well as the
shock-capturing ability of the  scheme proposed in this work.

\subsection{Smooth Alfv\'en wave problem}
We consider two variants of the so-called
smooth Alfv\'en wave problem: a (1) 2.5D
variant and (2) 3D variant.
Note that the term 2.5D here refers to the case of a two-dimensional computational domain, 
but with vector unknowns $\u$ and $\B$ that have three non-trivial
components.
In 2.5D for the magnetic potential we solve equation 
(\ref{eqn:system_compact}) with $\Av_{,z} = 0$ using the method described in
Section \ref{section:potential_space}.

\subsubsection{2.5D problem}
In this case the analytical solution
consists of a sinusoidal wave propagating at constant speed in
direction ${\bf n} = (\cos \phi, \sin \phi, 0)$ without changing shape. We
use $\phi = \arctan (0.5)$ and solve in the domain
$(x,y) \in [0,(\cos \phi)^{-1}] \times [0, (\sin \phi)^{-1}]$ with double periodic
boundary conditions. The initial values of all components are
described in \cite{article:HeRoTa10}. Similar problems were
also considered by Rossmanith \cite{article:Ro04b} and T\'oth \cite{article:To00}.

In Table \ref{table:alfen_cart} we show the results of a numerical
convergence study of the new constrained transport algorithm on a 2D Cartesian grid. We obtain full third order convergence
rates in all conserved quantities, as well as all the magnetic vector potential
components.

\begin{table}
\begin{center}
  \begin{tabular}{|c||c|c|c|c|c||}
    \hline
    &
    {\normalsize $\rho$}&
    {\normalsize $\rho u^1$}&
    {\normalsize $\rho u^2$}&
    {\normalsize $\rho u^3$}&
    {\normalsize ${\mathcal E}$} \\    
    \hline \hline 
    {\normalsize $32 \times 64$}  &
    {\normalsize $5.128 \times 10^{-4}$}  &  
    {\normalsize $2.413 \times 10^{-4}$}  &
    {\normalsize $6.059 \times 10^{-4}$}  &
    {\normalsize $5.142 \times 10^{-4}$}  &      
    {\normalsize $1.301 \times 10^{-4}$}   \\
    \hline
    {\normalsize $64 \times 128$}  &  
    {\normalsize $6.495 \times 10^{-5}$}  &  
    {\normalsize $3.097 \times 10^{-5}$}  &
    {\normalsize $7.548 \times 10^{-5}$}  &
    {\normalsize $6.469 \times 10^{-5}$}  &      
    {\normalsize $1.653 \times 10^{-5}$}   \\
    \hline
    {\normalsize $128 \times 256$}  &  
    {\normalsize $8.150 \times 10^{-6}$}  &  
    {\normalsize $3.924 \times 10^{-6}$}  &
    {\normalsize $9.412 \times 10^{-6}$}  &
    {\normalsize $8.099 \times 10^{-6}$}  &      
    {\normalsize $2.075 \times 10^{-6}$}   \\
    \hline
    {\normalsize $256 \times 512$}  &  
    {\normalsize $1.020 \times 10^{-6}$}  &  
    {\normalsize $4.933 \times 10^{-7}$}  &
    {\normalsize $1.176 \times 10^{-6}$}  &
    {\normalsize $1.014 \times 10^{-6}$}  &      
    {\normalsize $2.599 \times 10^{-7}$}   \\    
    \hline \hline 
    {\normalsize {\bf EOC}} & 
    {\normalsize 2.998}  & 
    {\normalsize 2.992}  & 
    {\normalsize 3.000}  & 
    {\normalsize 2.998}  & 
    {\normalsize 2.997}  \\
     \hline
  \end{tabular}

\bigskip

   \begin{tabular}{|c||c|c|c|}
    \hline
    &
    {\normalsize $B^1$} & 
    {\normalsize $B^2$} & 
    {\normalsize $B^3$} \\
    \hline \hline
    {\normalsize $32 \times 64$}  &  
    {\normalsize $2.877 \times 10^{-4}$} & 
    {\normalsize $5.754 \times 10^{-4}$}  &  
    {\normalsize $5.123 \times 10^{-4}$} \\
    \hline
    {\normalsize $64 \times 128$}  &  
    {\normalsize $3.583 \times 10^{-5}$} & 
    {\normalsize $7.166 \times 10^{-5}$}  &  
    {\normalsize $6.437 \times 10^{-5}$} \\
\hline
    {\normalsize $128 \times 256$}  &
    {\normalsize $4.464 \times 10^{-6}$} & 
    {\normalsize $8.928 \times 10^{-6}$}  &  
    {\normalsize $8.057 \times 10^{-6}$} \\
\hline
    {\normalsize $256 \times 512$}  &
    {\normalsize $5.581 \times 10^{-7}$} & 
    {\normalsize $1.116 \times 10^{-6}$}  &  
    {\normalsize $1.008 \times 10^{-6}$} \\
    \hline \hline 
   {\normalsize {\bf EOC}} & 
   {\normalsize 3.000} & 
   {\normalsize 3.000} & 
   {\normalsize 2.999} \\
 \hline
  \end{tabular}

\bigskip
  \begin{tabular}{|c||c|c|c|}
    \hline
   & 
    {\normalsize $\A^1$} & 
    {\normalsize $\A^2$} & 
    {\normalsize $\A^3$} \\
    \hline \hline
    {\normalsize $32 \times 64$}  &  
    {\normalsize $3.583 \times 10^{-5}$} & 
    {\normalsize $6.921 \times 10^{-5} $}  &  
    {\normalsize $1.053 \times 10^{-4}$} \\
    \hline
    {\normalsize $64 \times 128$}  &  
    {\normalsize $4.489 \times 10^{-6}$} & 
    {\normalsize $8.628 \times 10^{-6} $}  &  
    {\normalsize $1.327 \times 10^{-5}$} \\
\hline
    {\normalsize $128 \times 256$}  &
    {\normalsize $5.626 \times 10^{-7}$} & 
    {\normalsize $1.076 \times 10^{-6} $}  &  
    {\normalsize $1.661 \times 10^{-6}$} \\
\hline
    {\normalsize $256 \times 512$}  &
    {\normalsize $7.041 \times 10^{-8}$} & 
    {\normalsize $1.344 \times 10^{-7} $}  &  
    {\normalsize $2.077 \times 10^{-7}$} \\
    \hline \hline 
   {\normalsize {\bf EOC}} & 
   {\normalsize 2.998} & 
   {\normalsize 3.000} & 
   {\normalsize 2.999} \\
 \hline
  \end{tabular}
\end{center}
\caption{\label{table:alfen_cart}Convergence study of the 2.5D smooth Alfv\'en wave problem
on a Cartesian mesh. The tables show the $L_1$-error at time $t=1$ 
in the different physical quantities computed using the constrained transport algorithm. The
experimental order of convergence (EOC) is computed by comparing the
error for the two finest grids. All simulations are performed using a CFL number of $0.5$.}
\end{table} 

We have also computed this test problem on a mapped grid, which is a
scaled version of a grid from
Colella et al.\ \cite{article:CDHM2011}. 
The mapping has the form
\begin{equation}\label{eqn:mapping-Colella}
T(x_c,y_c) = \left( \begin{array}{c}
x_c\\y_c\end{array}\right) + \beta \sin\left( \frac{2 \pi
x_c}{L} \right) \sin\left( \frac{2 \pi y_c}{M} \right) \left( \begin{array}{c}1\\1\end{array}\right), 
\end{equation}
where $(x_c, y_c)$ are the coordinates in computational space, $\beta
\in \mathbb{R}$ is a parameter which determines the smoothnes of the
grid, and $L$, $M$ describe the length of the domain in the $x$ and
$y$ direction, respectively.
Here we use $\beta = 0.1$, $L = (\cos \phi)^{-1}$ and  $M=(\sin \phi)^{-1}$.
The mapped grid is shown in
Panel (a) of Figure \ref{fig:mapped1}.
Table \ref{table:alfen_mapped} confirms the third order
convergence rate also for the mapped grid computation.
Note that the error on the mapped grid is only slightly larger
than the error on the Cartesian grid. 
However, we note that the least squares approach that we used for the piecewise polynomial
reconstruction leads to the correct order only on grids that do not have highly stretched cells (e.g., see Petrovskaya \cite{article:Petrovskaya2007}).

\begin{figure}
\begin{center}
(a)\includegraphics[scale=0.41]{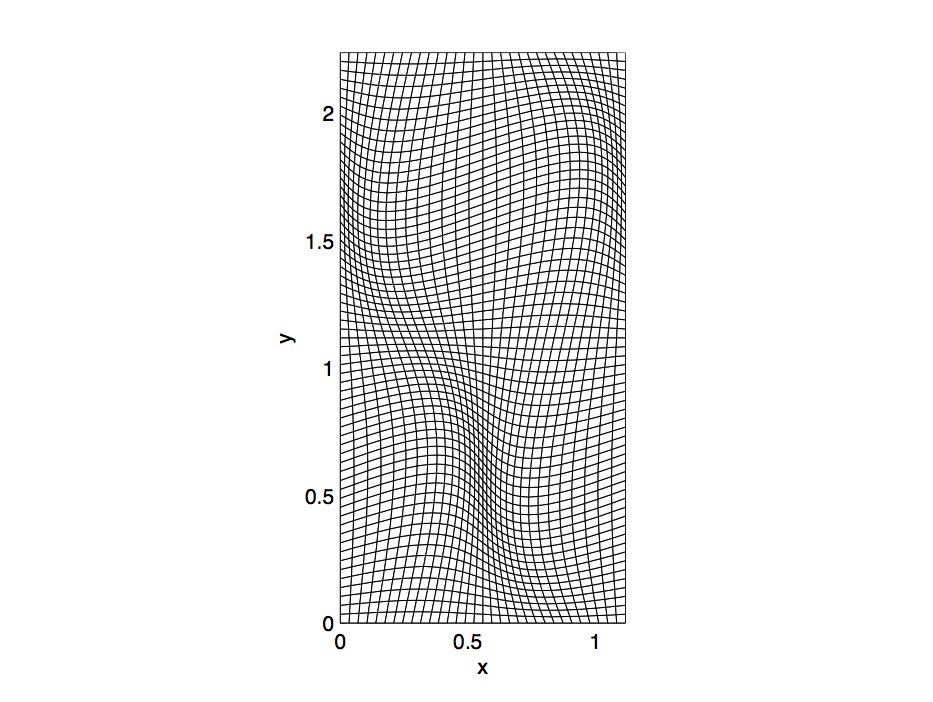} 
(b)\includegraphics[scale=0.41]{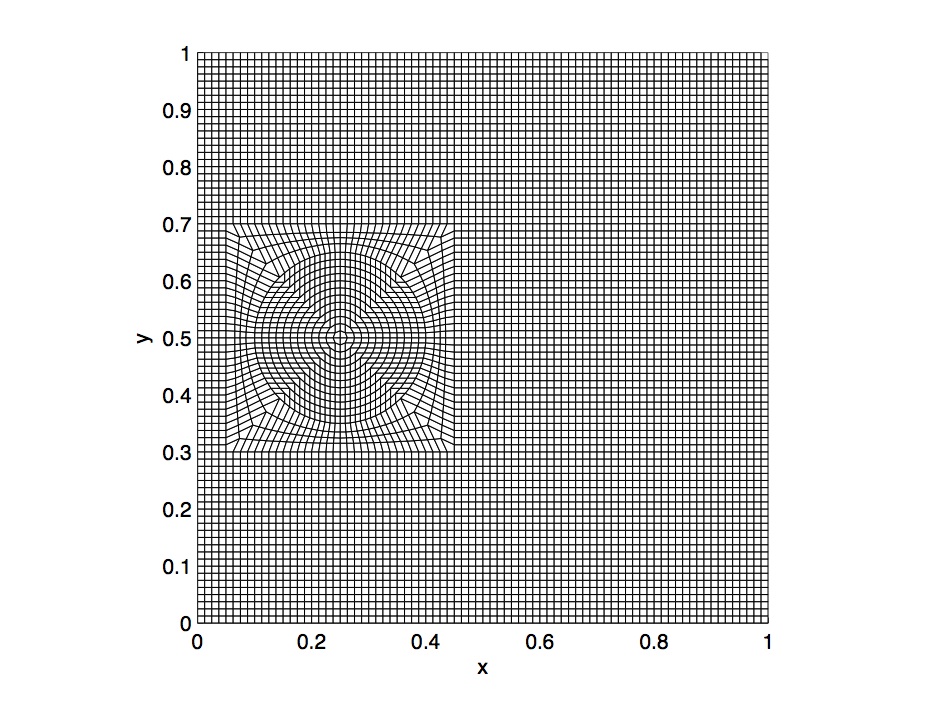}
\caption{\label{fig:mapped1}Shown in this figure are the mapped
grids used in (a) the convergence
  study of the smooth Alfv\'en wave problem and (b) the cloud-shock interaction problem.}
  \end{center}
\end{figure}

\begin{table}
\begin{center}
  \begin{tabular}{|c||c|c|c|c|c||}
    \hline
    &
    {\normalsize $\rho$}&
    {\normalsize $\rho u_1$}&
    {\normalsize $\rho u_2$}&
    {\normalsize $\rho u_3$}&
    {\normalsize ${\mathcal E}$} \\    
    \hline \hline 
    {\normalsize $32 \times 64$}  &
    {\normalsize $1.975 \times 10^{-3}$}  &  
    {\normalsize $7.167 \times 10^{-4}$}  &
    {\normalsize $1.223 \times 10^{-3}$}  &
    {\normalsize $7.711 \times 10^{-4}$}  &      
    {\normalsize $5.230 \times 10^{-4}$}   \\
    \hline
    {\normalsize $64 \times 128$}  &  
    {\normalsize $2.631 \times 10^{-4}$}  &  
    {\normalsize $9.756 \times 10^{-5}$}  &
    {\normalsize $1.557 \times 10^{-4}$}  &
    {\normalsize $9.745 \times 10^{-5}$}  &      
    {\normalsize $6.955 \times 10^{-5}$}   \\
    \hline
    {\normalsize $128 \times 256$}  &  
    {\normalsize $3.333 \times 10^{-5}$}  &  
    {\normalsize $1.245 \times 10^{-5}$}  &
    {\normalsize $1.953 \times 10^{-5}$}  &
    {\normalsize $1.222 \times 10^{-5}$}  &      
    {\normalsize $8.820 \times 10^{-6}$}   \\
    \hline \hline 
    {\normalsize {\bf EOC}} & 
    {\normalsize 2.981}  & 
    {\normalsize 2.970}  & 
    {\normalsize 2.996}  & 
    {\normalsize 2.996}  & 
    {\normalsize 2.980}  \\
     \hline
  \end{tabular}

  \bigskip
  
   \begin{tabular}{|c||c|c|c|}
    \hline
    &
    {\normalsize $B^1$} & 
    {\normalsize $B^2$} & 
    {\normalsize $B^3$} \\
    \hline \hline
    {\normalsize $32 \times 64$}  &  
    {\normalsize $4.827 \times 10^{-4}$} & 
    {\normalsize $9.735 \times 10^{-4}$}  &  
    {\normalsize $9.483 \times 10^{-4}$} \\
    \hline
    {\normalsize $64 \times 128$}  &  
    {\normalsize $6.014 \times 10^{-5}$} & 
    {\normalsize $1.213 \times 10^{-4}$}  &  
    {\normalsize $1.219 \times 10^{-4}$} \\
\hline
    {\normalsize $128 \times 256$}  &
    {\normalsize $7.511 \times 10^{-6}$} & 
    {\normalsize $1.513 \times 10^{-5}$}  &  
    {\normalsize $1.534 \times 10^{-5}$} \\
    \hline \hline 
   {\normalsize {\bf EOC}} & 
   {\normalsize 3.001} & 
   {\normalsize 3.002} & 
   {\normalsize 2.990} \\
 \hline
  \end{tabular}
 
  \bigskip
  
  \begin{tabular}{|c||c|c|c|}
    \hline
    &
    {\normalsize $\A^1$} & 
    {\normalsize $\A^2$} & 
    {\normalsize $\A^3$} \\
    \hline \hline
    {\normalsize $32 \times 64$}  &  
    {\normalsize $7.355 \times 10^{-5}$} & 
    {\normalsize $1.260 \times 10^{-4} $}  &  
    {\normalsize $1.749 \times 10^{-4}$} \\
    \hline
    {\normalsize $64 \times 128$}  &  
    {\normalsize $9.358 \times 10^{-6}$} & 
    {\normalsize $1.595 \times 10^{-5} $}  &  
    {\normalsize $2.216 \times 10^{-5}$} \\
\hline
    {\normalsize $128 \times 256$}  &
    {\normalsize $1.175 \times 10^{-6}$} & 
    {\normalsize $2.000 \times 10^{-6} $}  &  
    {\normalsize $2.776 \times 10^{-6}$} \\
    \hline \hline 
   {\normalsize {\bf EOC}} & 
   {\normalsize 2.994} & 
   {\normalsize 2.996} & 
   {\normalsize 2.997} \\
 \hline
  \end{tabular}
  \end{center}
  \caption{\label{table:alfen_mapped}
Convergence study of the 2.5D smooth Alfv\'en wave problem
on a mapped grid. The tables show the $L_1$-error at time $t=1$ 
in the different physical quantities computed using the constrained transport algorithm. The
experimental order of convergence (EOC) is computed by comparing the
error for the two finest grids. All simulations are performed using a CFL number of $0.5$.}
\end{table}

\subsubsection{3D problem}
Next we have performed a convergence study for the
3D smooth Alfv\'en wave problem on a Cartesian grid. The
results, which again confirm the third order convergence rate, are
presented in Table \ref{table:3dconv_study}.
The initial values and the computational domain for this problem are 
described in \cite{article:HeRoTa10}.
\begin{table}
\begin{center}
  \begin{tabular}{|c||c|c|c|c|c||}
    \hline
    &
    {\normalsize $\rho$}&
    {\normalsize $\rho u_1$}&
    {\normalsize $\rho u_2$}&
    {\normalsize $\rho u_3$}&
    {\normalsize ${\mathcal E}$} \\    
    \hline \hline 
    {\normalsize $16 \times 32 \times 32$}  &
    {\normalsize $2.055 \times 10^{-3}$}  &  
    {\normalsize $9.992 \times 10^{-4}$}  &
    {\normalsize $1.716 \times 10^{-3}$}  &
    {\normalsize $1.911 \times 10^{-3}$}  &      
    {\normalsize $5.133 \times 10^{-4}$}   \\
    \hline
    {\normalsize $32 \times 64 \times 64$}  &  
    {\normalsize $2.684 \times 10^{-4}$}  &  
    {\normalsize $1.343 \times 10^{-4}$}  &
    {\normalsize $2.213 \times 10^{-4}$}  &
    {\normalsize $2.532 \times 10^{-4}$}  &      
    {\normalsize $6.611 \times 10^{-5}$}   \\
    \hline
    {\normalsize $64 \times 128 \times 128$}  &  
    {\normalsize $3.419 \times 10^{-5}$}  &  
    {\normalsize $1.679 \times 10^{-5}$}  &
    {\normalsize $2.764 \times 10^{-5}$}  &
    {\normalsize $3.110 \times 10^{-5}$}  &      
    {\normalsize $8.701 \times 10^{-6}$}   \\
    \hline \hline 
    {\normalsize {\bf EOC}} & 
    {\normalsize } 2.973& 
    {\normalsize } 3.000& 
    {\normalsize } 3.001&
    {\normalsize } 3.025&
    {\normalsize } 2.926\\
 \hline
  \end{tabular}

  \bigskip
  
   \begin{tabular}{|c||c|c|c|}
    \hline
     &
    {\normalsize {\bf $L_1$ Error in }$B^1$} & 
    {\normalsize {\bf $L_1$  Error in }$B^2$} & 
    {\normalsize {\bf $L_1$ Error in }$B^3$} \\
    \hline \hline
    {\normalsize $16 \times 32 \times 32$}  &  
    {\normalsize $1.231 \times 10^{-3}$} & 
    {\normalsize $1.805 \times 10^{-3}$}  &  
    {\normalsize $1.778 \times 10^{-3}$} \\
    \hline
    {\normalsize $32 \times 64 \times 64$}  &  
    {\normalsize $1.581 \times 10^{-4}$} & 
    {\normalsize $2.304 \times 10^{-4}$}  &  
    {\normalsize $2.296 \times 10^{-4}$} \\
\hline
    {\normalsize $64 \times 128 \times 128$}  &
    {\normalsize $1.957 \times 10^{-5}$} & 
    {\normalsize $2.869\times 10^{-5}$}  &  
    {\normalsize $2.818 \times 10^{-5}$} \\
    \hline \hline 
   {\normalsize {\bf EOC}} & 
   {\normalsize } 3.014& 
   {\normalsize } 3.006& 
   {\normalsize } 3.026\\
 \hline
  \end{tabular}
 
  \bigskip
  
  \begin{tabular}{|c||c|c|c|}
    \hline
     &
    {\normalsize {\bf $L_1$ Error in }$\A^1$} & 
    {\normalsize {\bf $L_1$  Error in }$\A^2$} & 
    {\normalsize {\bf $L_1$ Error in }$\A^3$} \\
    \hline \hline
    {\normalsize $16 \times 32 \times 32$}  &  
    {\normalsize $1.588 \times 10^{-4}$} & 
    {\normalsize $2.933 \times 10^{-4} $}  &  
    {\normalsize $2.869 \times 10^{-4}$} \\
    \hline
    {\normalsize $32 \times 64 \times 64$}  &  
    {\normalsize $2.340 \times 10^{-5}$} & 
    {\normalsize $3.642 \times 10^{-5} $}  &  
    {\normalsize $3.671 \times 10^{-5}$} \\
\hline
    {\normalsize $64 \times 128 \times 128$}  &
    {\normalsize $2.550 \times 10^{-6}$} & 
    {\normalsize $4.780 \times 10^{-6} $}  &  
    {\normalsize $4.617 \times 10^{-6}$} \\
    \hline \hline 
   {\normalsize {\bf EOC}} & 
   {\normalsize } 3.198& 
   {\normalsize } 2.930& 
   {\normalsize } 2.991\\
 \hline
  \end{tabular}
  \caption{\label{table:3dconv_study}
Convergence study of the 3D  smooth Alfv\'en wave problem
on a Cartesian grid. The tables show the $L_1$-error at time $t=1$ 
in the different physical quantities computed using the constrained transport algorithm. The
experimental order of convergence (EOC) is computed by comparing the
error for the two finest grids. All simulations are performed using a CFL number of $0.6$.}
\end{center}
\end{table}

\subsection{2.5D shock tube problem on mapped grids}
Next we consider a 1D shock tube problem on a scaled versions of the
2D mapped grid given by (\ref{eqn:mapping-Colella}). 
This example demonstrates the performance of the new CT method on
shocks not moving along grid lines. 

The computational domain is the square $[-0.7, 0.7]\times [-0.7,0.7]$. 
In order to avoid inaccuracies at the boundary of the domain, we restrict the
mapping to the area $-0.6 \leq x_c, y_c \le 0.6$. 
The rest of the mesh is Cartesian. We set $L=M=1.2$ and use different
values of $\beta$.

The Riemann initial data have the form
\begin{equation}
\begin{split}
&\left(\rho,\, u^1,\, u^2, \, u^3, \, p, \, B^1, \, B^2, \, B^3 \right)(0, \x) \\
&=\left\{ \begin{array}{l c c}
\Bigl(1.08, \, 1.2, \, 0.01, \, 0.5, \, 0.95, \, \frac{2}{\sqrt{4\pi}}, \, \frac{3.6}{\sqrt{4\pi}}, \, 
\frac{2}{\sqrt{4\pi}} \Bigr)
& \text{if} & x < 0,\\
\Bigl( 1,\, 0,\, 0,\, 0,\, 1,\, \frac{2}{\sqrt{4\pi}}, \, \frac{4}{\sqrt{4\pi}}, \, \frac{2}{\sqrt{4 \pi}} \Bigr) & \text{if} &  x \ge 0.
\end{array} \right.
\end{split}
\end{equation}
As the initial condition for the magnetic potential we use
\begin{equation}
\left({\A}^1, \, {\A}^2, \, {\A}^3 \right)(0,\x)
=\left(0, \, x B^3, \, y B^1 - x B^2 \right).
\end{equation}
We utilize zeroth order extrapolation on all boundaries for the MHD variables and linear extrapolation for the magnetic potential.
Shown is a comparison of scatter plots for the magnetic field components along the $x$-axis.

In Figure \ref{fig:shock-tube}, we show results for the magnetic field components using the new CT
method. Here we compare results for differnt grids, namely a Cartesian
grid and two different versions of the mapped grid obtained by setting
the parameter $\beta$ to either $1/50$ or $1/15$. For all simulations
we used grids with $200 \times 200$ mesh cells.
Note that for the mapped grid computations the solution structure,
moving only in $x$-direction, is not aligned with the grid.
It has been observed previously that this leads to unphysical
oscillations, see \cite{article:HeRoTa10,article:MiTz10}.
A numerical convergence study for $B^1$ on mapped grids with different
resolution showed that the error in the $L_1$ norm decreases with a
rate of about 1/3.
By introducing more numerical viscosity (i.e.\ decreasing the value of
$\lambda_{ii}$) these oscillations can be slightly reduced. However, they
could not be avoided. Here we show results for $\lambda_{ii}=10^3$.
The solid lines in these plots are obtained by computing solutions of the one-dimensional Riemann problem on a very fine equidistant mesh.


\begin{figure}[htb]
\begin{center}
\includegraphics[width=0.32\textwidth]{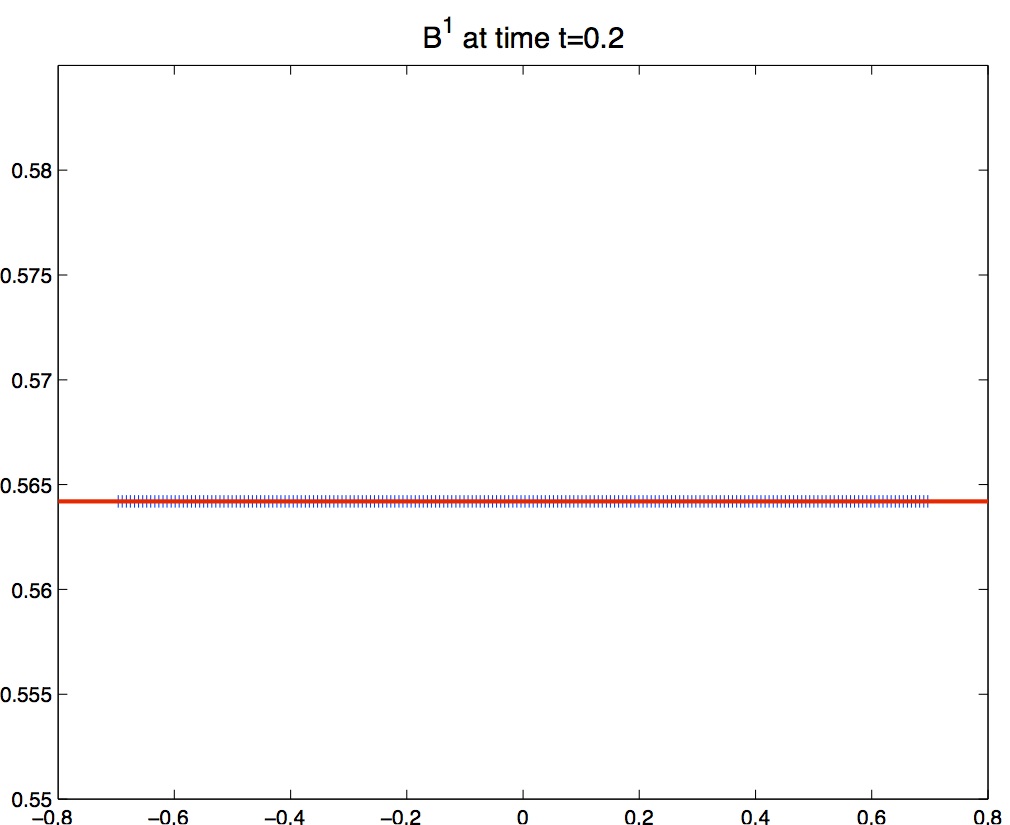}\includegraphics[width=0.32\textwidth]{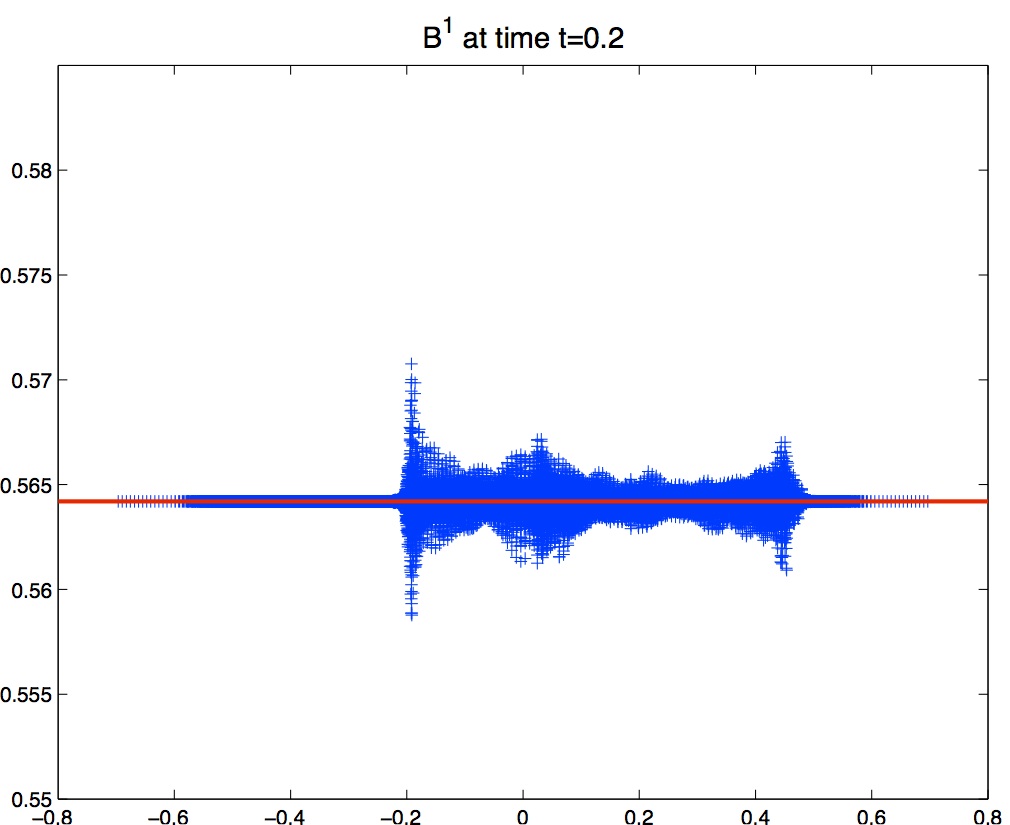}\includegraphics[width=0.32\textwidth]{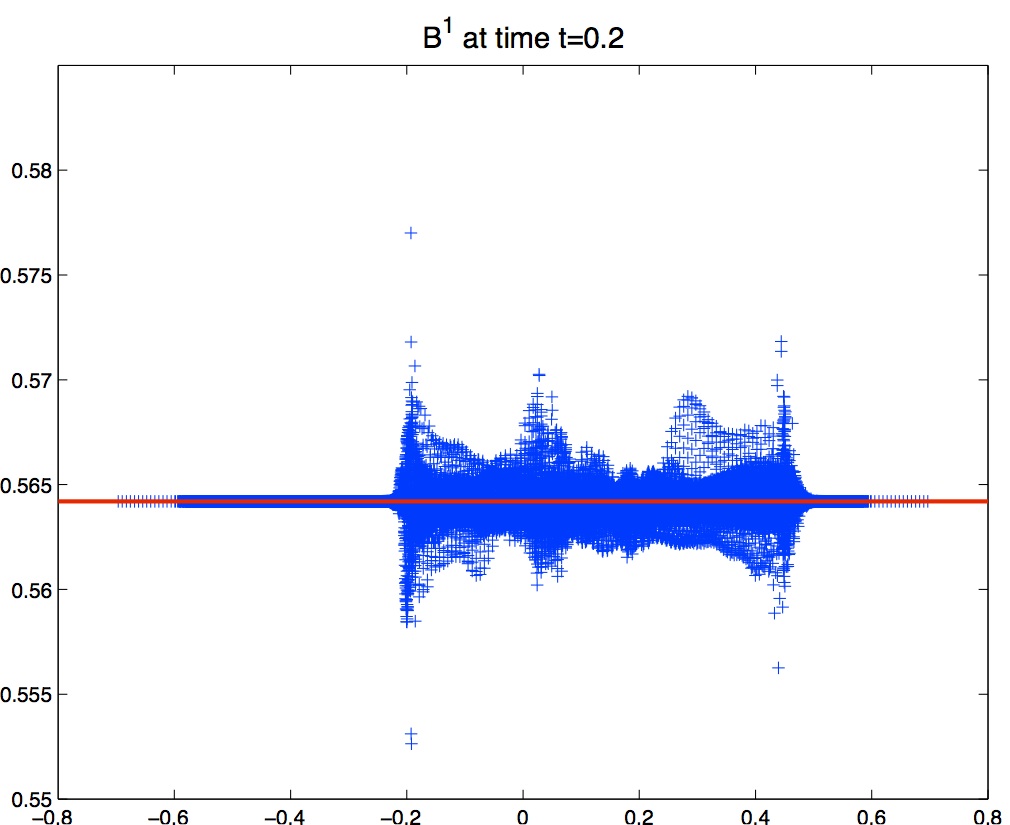}

\includegraphics[width=0.32\textwidth]{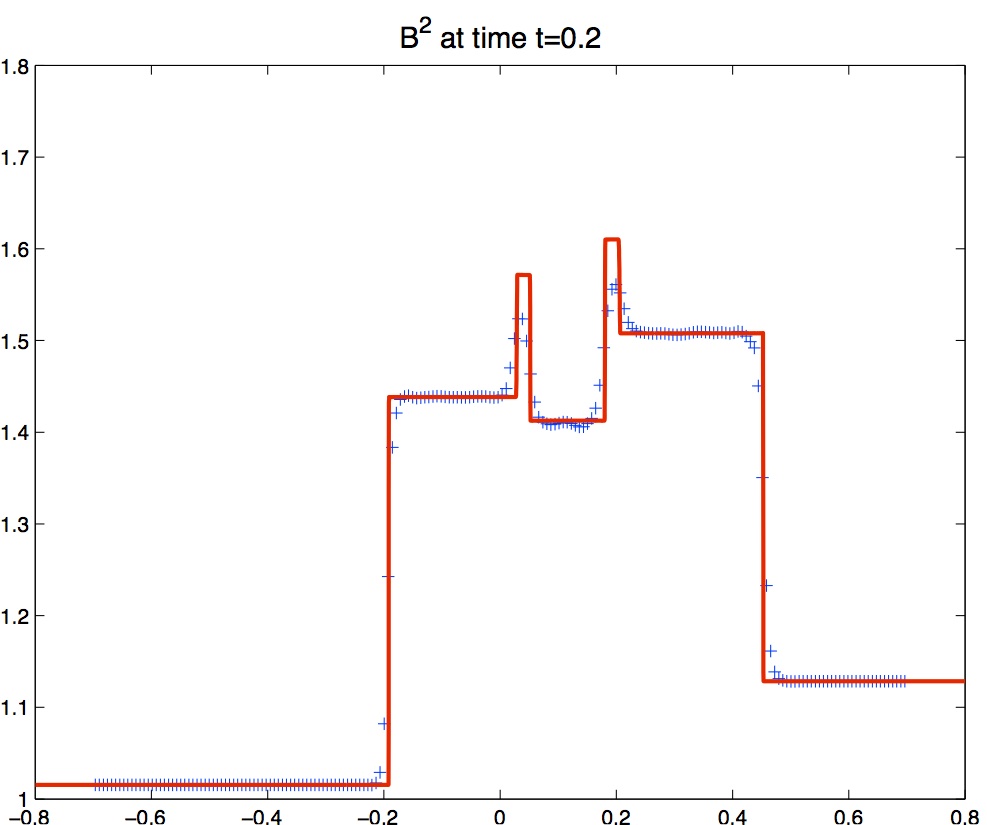}\includegraphics[width=0.32\textwidth]{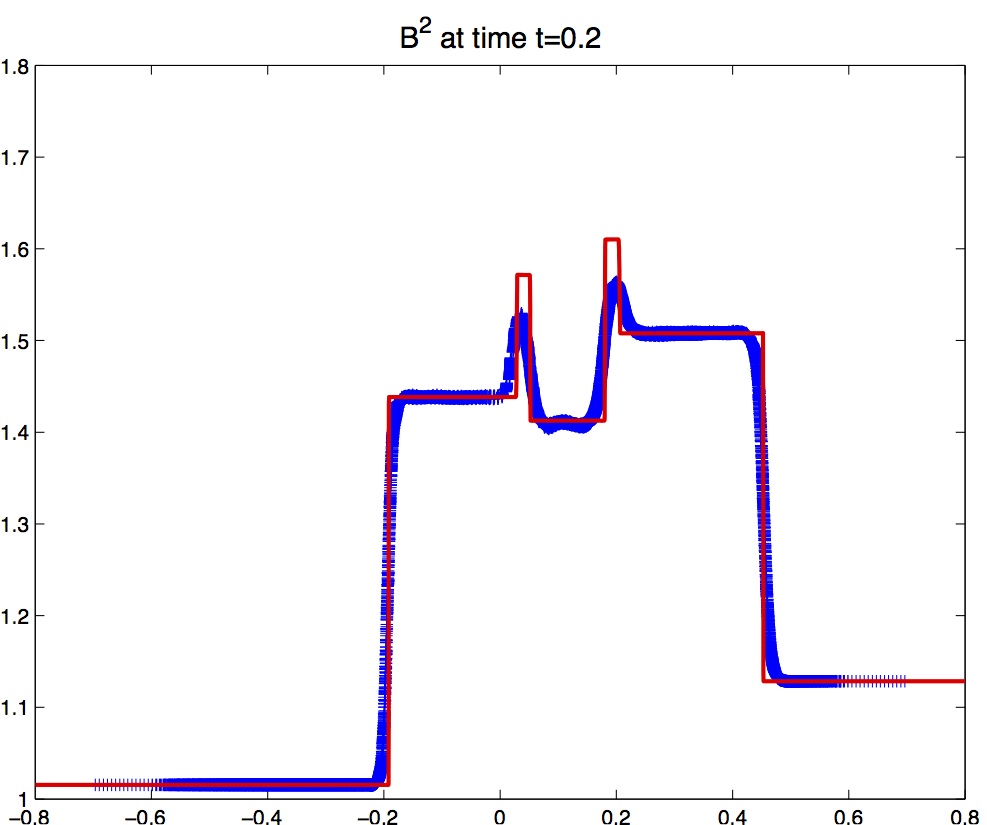}\includegraphics[width=0.32\textwidth]{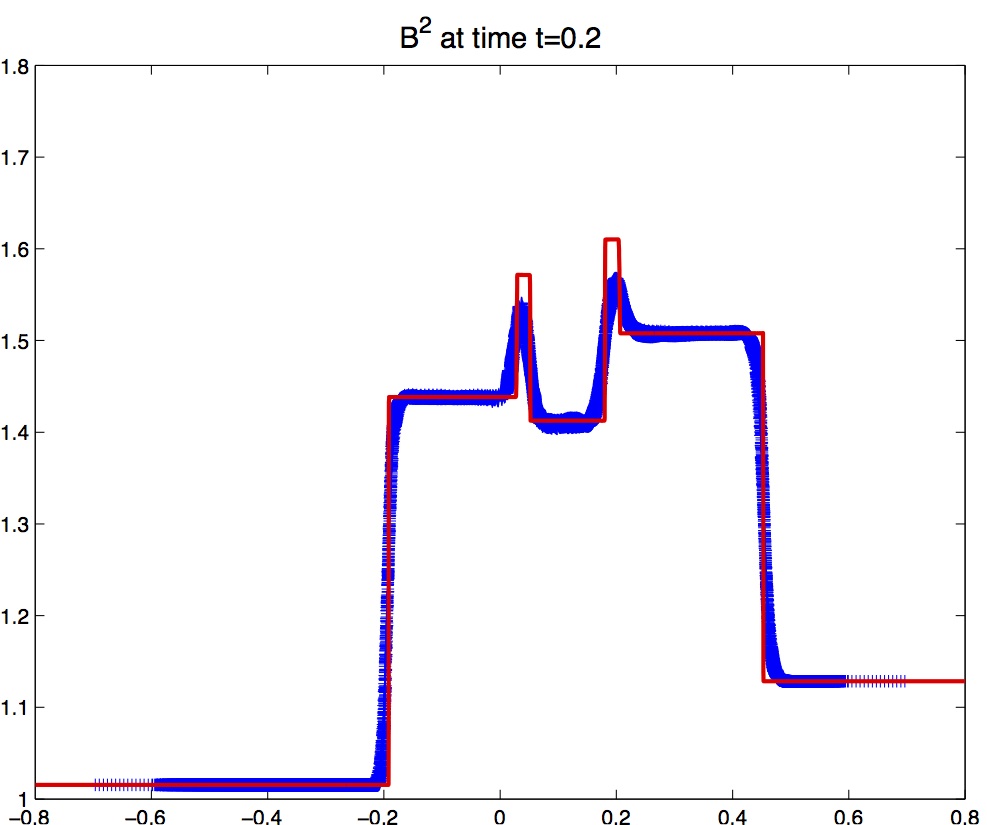}

\includegraphics[width=0.32\textwidth]{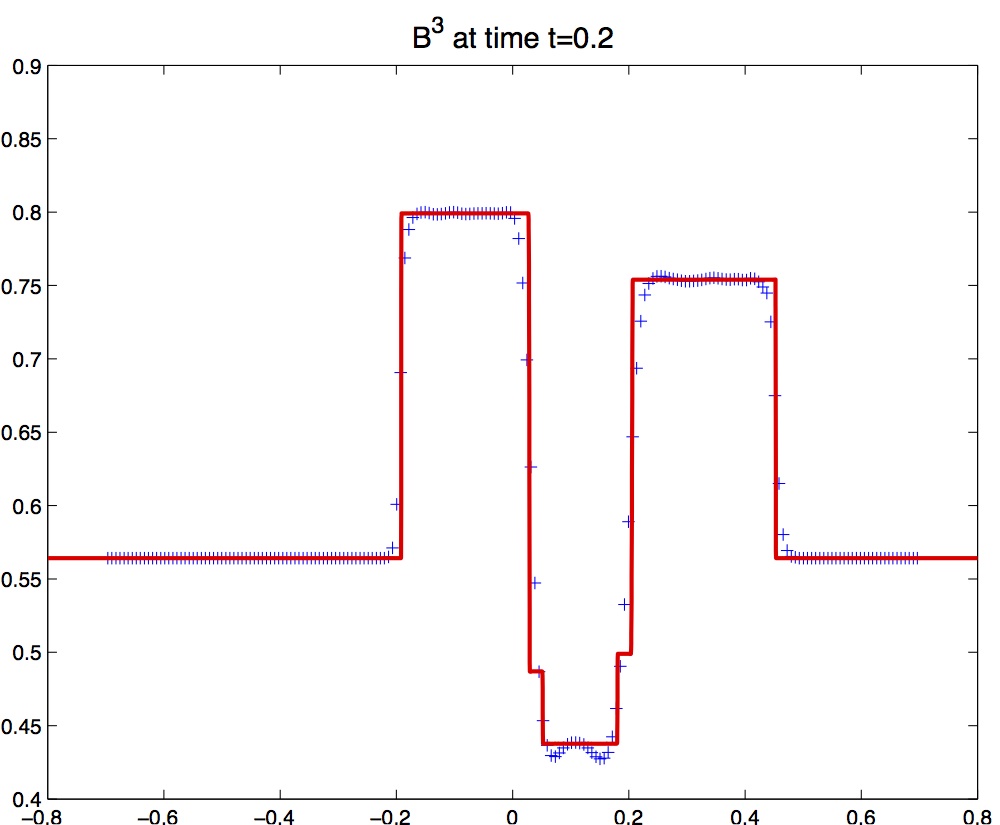}\includegraphics[width=0.32\textwidth]{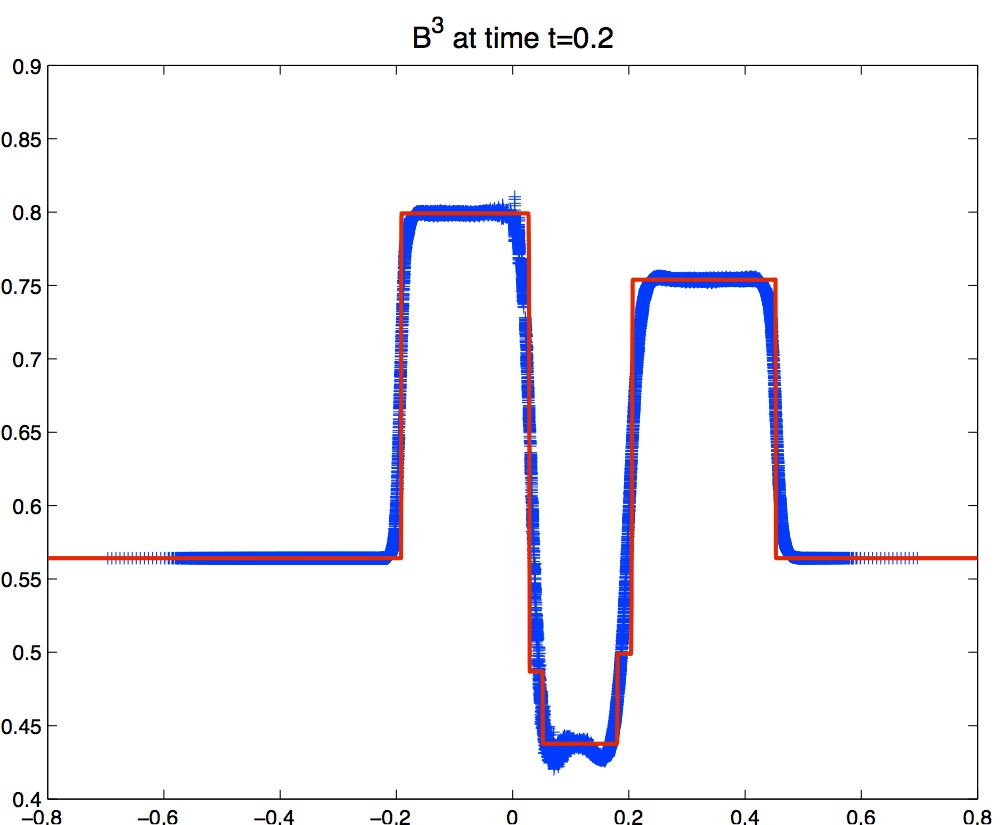}\includegraphics[width=0.32\textwidth]{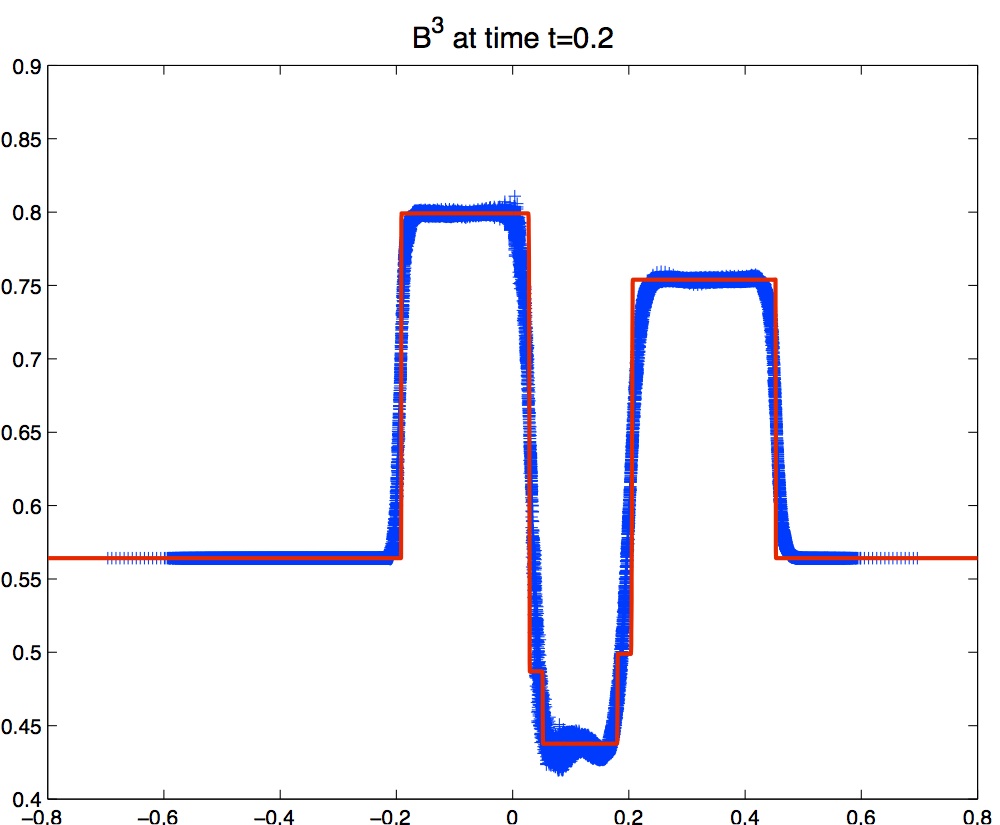}
\caption{\label{fig:shock-tube}Shown are the magnetic field components for the shock tube
  problem at $t=0.2$. The left column shows the results for a
  Cartesian grid computation; the middle column shows results on a
  mapped grid which is a small perturbation of the Cartesian grid
  (using $\beta = 1/50$); the
  right column shows results on a mapped grid which is a larger
  perturbation of the Cartesian grid (using $\beta = 1/15$).}   
  \end{center}
\end{figure}

\subsection{Cloud-shock interaction problem}
Next we present numerical results of the proposed method for a
cloud-shock interaction problem.  Again we consider 3D and 2.5D
variants of this problem.

\subsubsection{3D problem}
The initial conditions consist of a
shock that is located at $x=0.05$, with
\begin{equation}
\begin{split}
&\left(\rho, \, u^1, \, u^2, \, u^3,  \, p,  \, B^1, \, B^2, \, B^3 \right)(0, \x) \\
&=\biggl\{ \begin{array}{l c c}
\left(3.86859,  \, 11.2536,  \, 0,  \, 0,  \, 167.345,  \, 0,  \, 2.1826182, 
 \, -2.1826182\right)
& \text{if} &x < 0.05,\\
\left( 1,  \, 0,  \, 0,  \, 0,  \, 1,  \, 0,  \, 0.56418958,  \, 0.56418958\right) & \text{if} & x\ge 0.05,
\end{array}
\end{split}
\end{equation}
and a spherical cloud of density $\rho = 10$ with radius $r=0.15$ 
centered at $(0.25,0.5,0.5)$. The cloud is in hydrostatic
equilibrium with the fluid to the right of the shock.
The initial conditions for the magnetic potential are given by
\begin{equation}
\Av (0, \x) = \left\{ \begin{array}{l c c}
\left( 2.1826182  \,  y,  \, 0,  \,  -2.1826182 \,  (x-0.05) \right)^T & \text{if} & x < 0.05,
\\
\left( -0.56418958 \,  y,  \, 0,  \, 0.56418956 \, (x-0.05) \right)^T & \text{if} & x \ge 0.05.
\end{array}\right.
\end{equation} 
The computational domain is the unit cube. Inflow boundary conditions
are used at the left side and outflow boundary conditions are used at
all other sides.
Without the constrained transport step the method would fail to
compute the solution structure. In Figure \ref{fig:cloud-shock3d} we
present results of a computation with the three-dimensional method
proposed in this paper. Compare also with \cite{article:HeRoTa10},
where this problem was computed with our previous approach. 

\begin{figure}[htb!]
\begin{center}
\centerline{\includegraphics[scale=0.85]{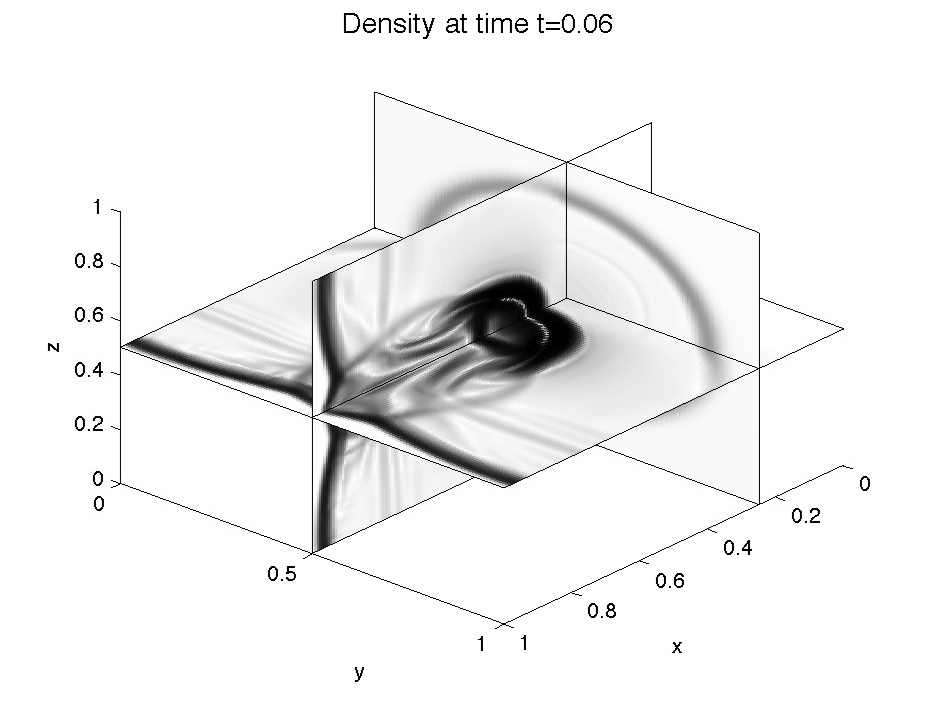}}
\caption{This figure shows the density at time
$t=0.06$ on a three-dimensional Cartesian grid using
$150\times150\times150$ mesh cells. For this simulation we used a
$2^{\text{nd}}$ order version of the proposed algorithm.\label{fig:cloud-shock3d}}
\end{center}
\end{figure}

\subsubsection{2.5D problem}
We also studied the cloud-shock interaction problem in the 2.5D
 case. 
The initial values are taken from the
3D test problem by setting $z=0.5$. 
In the 2.5D case we show results for $B^3$ as computed
by the new constrained transport algorithm that updates all three
components of the magnetic field. The results compare well with those
obtained by our previous approach \cite{article:HeRoTa10} and also with the results from the
2D unsplit method of \cite{article:Ro04b} in which only $B^1$
and $B^2$ are updated by a constrained transport method.
We compute the solution both on a Cartesian grid as well as on a
mapped grid. For the mapped grid computation we use a grid of the form
shown in Panel (b) of Figure \ref{fig:mapped1}. This is a small modification of a
grid with circular inclusion discussed in \cite{article:CHL08}.
The initial cloud position is inside
the circular region of the mapped grid. The results on both grids
compare well, although the solution structure is slightly sharper
resolved on the Cartesian mesh (see Figure \ref{fig:2dcloudshock}).

\begin{figure}[htb!]
\begin{center}
(a)\includegraphics[scale=0.4]{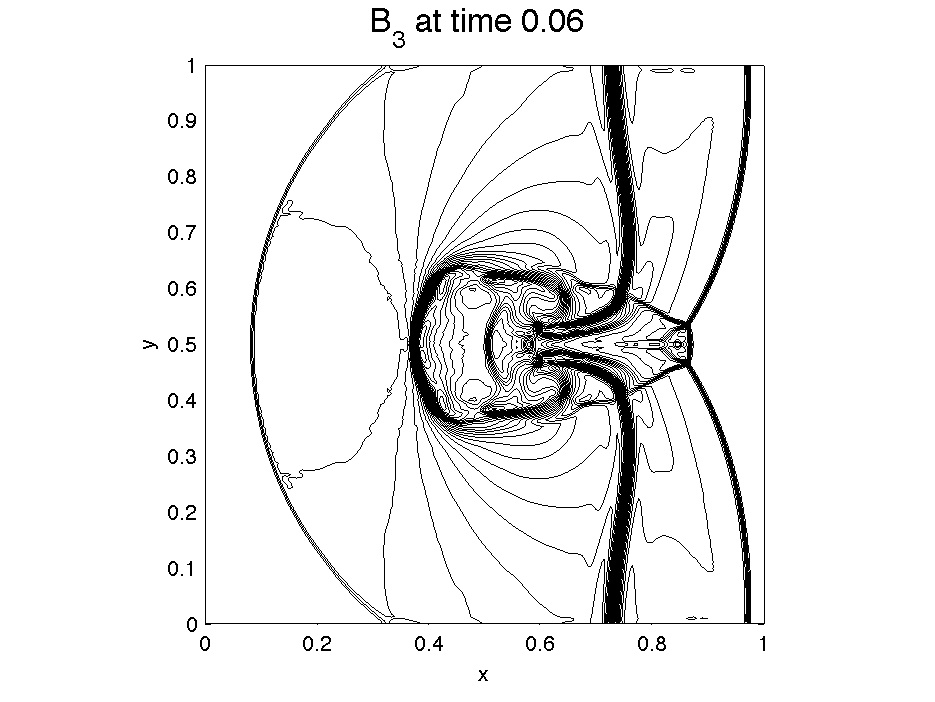}\hfill
(b)\includegraphics[scale=0.4]{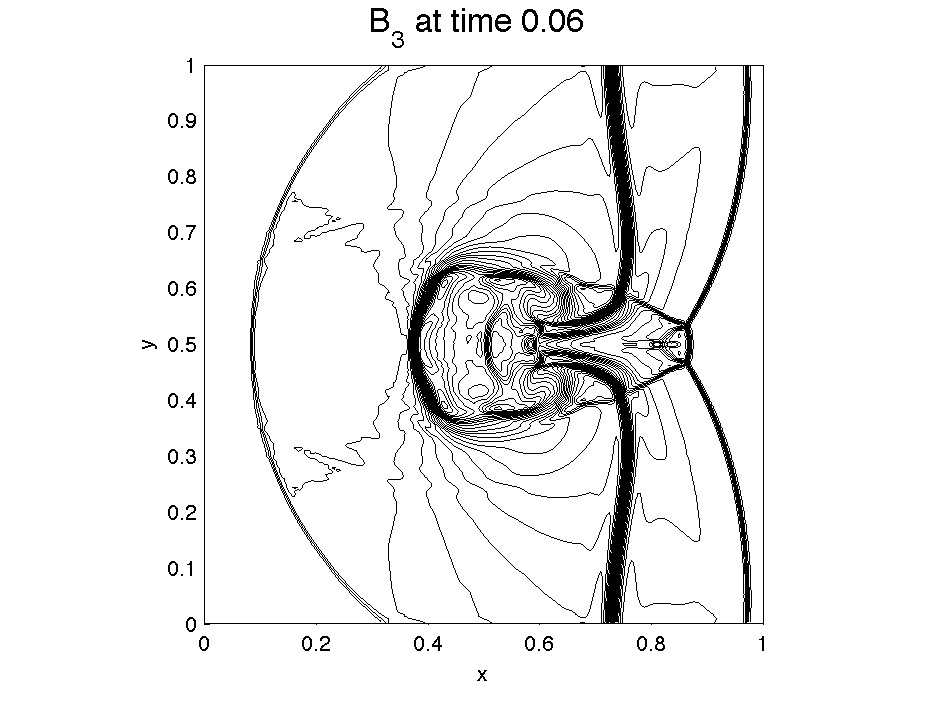}
\caption{\label{fig:2dcloudshock}The $3^{\text{rd}}$ component of the magnetic field at time
$t=0.06$  (a) on a $256 \times 256$ Cartesian grid and (b) on the
mapped grid shown in Figure \ref{fig:mapped1} (right) using the $3^{\text{rd}}$ order version of the algorithm. 
Here the 2.5 dimensional problem was implemented in such a way that 
all three components of the magnetic field were updated in the constrained transport step.}
\end{center}
\end{figure}

\section{Conclusions}
\label{sec:conclusions}
In this work we developed a finite volume
method for solving the 2D and 3D ideal MHD equations on both
Cartesian and logically Cartesian mapped grids. The scheme is
a method of lines that is
based on a  finite volume discretization
in space coupled with a strong-stability-preserving Runge-Kutta
time stepping method. 
By using this method of lines discretization, we were able to construct a
third-order accurate algorithm for the MHD equations. 
In order to control errors in the divergence of the magnetic
field we developed
a novel constrained transport approach that  couples
to the finite volume discretization. 
Our constrained transport methodology allowed us to overcome two important difficulties:
(1) the evolution equation for the magnetic potential is only
weakly hyperbolic, and (2) standard limiters applied to the
magnetic potential do not adequately control unphysical
oscillations in the magnetic field.

The method described in this work is based on the following procedure
that is carried out during each Runge-Kutta stage:
\begin{enumerate}
\item Update the MHD variables without regard for controlling
	discrete $\nabla \cdot \B$ errors. The updated magnetic field
	in this step is the {\it predicted} magnetic field. This step
	is carried out with the finite volume schemes described
	in Section \ref{section:mhd_space}.
\item Update the magnetic vector potential by solving a weakly
hyperbolic vector transport equation. This transport equation
is simply the induction equation with ideal Ohm's law, but written
in potential form using the Weyl gauge. This step
	is carried out with the non-conservative finite volume schemes described
	in Section \ref{section:potential_space}. Special artificial
	resistivity limiters described in Section \ref{sec:special_limiters} are used in this step
	to simultaneously control unphysical oscillations in the
	magnetic potential and the magnetic field.
\item Define the {\it corrected} magnetic field as the result of
taking the discrete curl of the updated magnetic vector potential.
The precise form of the discrete curl operation is described in Section \ref{sec:curlA}.
\end{enumerate}
The resulting scheme was applied to several 2.5D and 3D test cases
on both Cartesian and mapped grids. These test cases
 demonstrated two important features: (1) we are able to obtain full 
third-order accuracy on smooth problems and (2) we are 
able to accurately capture shock waves. 

Finally we note that although there are some free parameters required in the definition of the
artificial resistivity added in the magnetic vector potential equation update,
we have found parameter values that seem to robustly work for a large 
class of problems. In all the simulations presented in this work we have
stuck to these same parameter values. Experimentation with the $\lambda_{ii}$ 
parameter in \eqref{eqn:artvisc_params} shows that one can slightly increase or reduce
the amount of artificial resistivity, but that for a broad range of $\lambda_{ii}$ the results
are qualitatively the same. 

\bigskip

\noindent
{\bf Acknowledgements.}
This work was supported in part by the DFG
through FOR1048 and the NSF grants DMS--0711885 and 
DMS--1016202.

\bibliographystyle{plain}

\end{document}